\documentclass{amsart}

\usepackage{amssymb}
\usepackage{amsthm}

\usepackage{epsfig}  		% For postscript
\usepackage{subfig}
\usepackage{graphicx}

\newtheorem{theorem}{Theorem}[section]

\newtheorem{lemma}[theorem]{Lemma}
\newtheorem{corollary}[theorem]{Corollary}

\numberwithin{equation}{section}

\theoremstyle{definition}
\newtheorem{definition}[theorem]{Definition}

\theoremstyle{remark}

\newcommand{\bD}{\mathbb{D}}
\newcommand{\bC}{\mathbb{C}}
\newcommand{\bR}{\mathbb{R}}

\newcommand{\bN}{\mathbb{N}}
\newcommand{\bZ}{\mathbb{Z}}

\newcommand{\cF}{\mathcal{F}}
\newcommand{\cJ}{\mathcal{J}}

\newcommand{\cQ}{\mathcal{Q}}
\newcommand{\ra}{\rightarrow}
\newcommand{\pa}{\partial}
\DeclareMathOperator{\diam}{diam}
\DeclareMathOperator{\Hdim}{Hdim}
\DeclareMathOperator{\Pdim}{Pdim}
\DeclareMathOperator{\UMdim}{\overline{Mdim}}
\DeclareMathOperator{\dist}{dist}
\begin{document}

%%
%% The title of the paper goes here.  Edit to your title.
%%

\title{Transcendental Julia Sets with Fractional Packing Dimension}

%%
%% Now edit the following to give your name and address:
%% 

\author{Jack Burkart}
\address{Department of Mathematics, Stony Brook University, Stony Brook, NY 11794}
\email{jack.burkart@stonybrook.edu}
%\urladdr{www.math.sc.edu/$\sim$howard} % Delete if not wanted.

%%
%% If there is another author uncomment and edit the following.
%%

%\author{Second Author}
%\address{Department of Mathematics, University of South Carolina,
%Columbia, SC 29208}
%\email{second@math.sc.edu}
%\urladdr{www.math.sc.edu/$\sim$second}

%%
%% If there are three of more authors they are added in the obvious
%% way. 
%%

%%%
%%% The following is for the abstract.  The abstract is optional and
%%% if not used just delete, or comment out, the following.
%%%

\begin{abstract}
We construct a family of transcendental entire functions whose Julia sets have packing dimension in $(1,2)$. These are the first examples where the computed packing dimension is not $1$ or $2$. Our construction will allow us further show that the set of packing dimensions attained is dense in the interval $(1,2)$, and that the Hausdorff dimension of the Julia sets can be made arbitrarily close to the corresponding packing dimension.
\end{abstract}

%%
%%  LaTeX will not make the title for the paper unless told to do so.
%%  This is done by uncommenting the following.
%%

 \maketitle
\newpage
%%
%% LaTeX can automatically make a table of contents.  This is done by
%% uncommenting the following:
%%

%\tableofcontents

%%
%%  To enter text is easy.  Just type it.  A blank line starts a new
%%  paragraph. 
%%

\section{Introduction}
Let $f: \bC \ra \bC$ be a transcendental (non-polynomial) entire function. We denote the \textit{$n$th iterate} of $f$ by $f^n$. We define the \textit{Fatou set}, $\cF(f)$, to be the set of all points so that $\{f^n\}_{n=1}^{\infty}$ locally forms a normal family. Thus the Fatou set is the ``stable'' set for the dynamics of $f$. We define the \textit{Julia set}, $\cJ(f)$, to be the complement of the Fatou set. This is the set where the dynamics of $f$ are chaotic. A primary aim of complex dynamics is to study the geometric and topological properties of the Julia set. We refer the reader to \cite{CG} and \cite{DS} and  for an introduction to complex dynamics the rational and transcendental setting, respectively.

In this paper we prove the following theorem.
\begin{theorem}
\label{Main1}
There exists a transcendental entire function $f: \bC \ra \bC$ such that the packing dimension of $\cJ(f) \in (1,2)$.
\end{theorem}
Our techniques generate a family of entire functions, and we actually have the following stronger result.
\begin{theorem}
\label{Main2}
The set of packing dimensions attained is dense in $(1,2)$. In particular, let $s \in (1,2)$ and $\epsilon_0 > 0$ be given. Then there exists a transcendental entire $f$ so that
$$s -\epsilon_0 \leq \Hdim(\cJ(f)) \leq \Pdim(\cJ(f)) \leq s + \epsilon_0.$$ 
\end{theorem}
	
In \cite{Baker}, Baker proved that the Julia set of a transcendental entire function must always contain a compact connected set, and it follows immediately that the Hausdorff dimension of the Julia set must always be greater than or equal to $1$. In \cite{Mis}, Misiurewicz showed that the Julia set of $e^z$ was the entire complex plane, and in \cite{Mc} McMullen showed that the Julia sets of the exponential and sine families of entire functions always have Hausdorff dimension $2$, but need not be all of $\bC$. These examples can also have positive or zero area measure. Reducing the dimension of the Julia set is therefore the difficult task in the transcendental setting, and in \cite{S1}, Stallard constructed examples in the Eremenko-Lyubich class that had Hausdorff dimension arbitrarily close to $1$, and refined this result further in \cite{S3} and \cite{S4} to include all values in $(1,2)$. Moreover, in \cite{S2}, Stallard showed that in the Eremenko-Lyubich class the Hausdorff dimension must be strictly greater than $1$. Recently, in \cite{CB1}, Bishop constructed a transcendental entire function with Julia set having Hausdorff dimension $1$. This example demonstrates that all values of Hausdorff dimension in $[1,2]$ can be achieved. 

Less is known about the packing dimension in the transcendental setting. In \cite{RS}, Rippon and Stallard show that if $f$ belongs to the Eremenko-Lyubich class, then the packing dimension of the Julia set is $2$. Bishop computed the packing dimension of the Julia set of his example above to be 1. Our result is the first of its kind where the computed packing dimension is strictly between $1$ and $2$. Packing dimension and other various dimensions relevant to the paper are defined in Section $3$. Figure $1$ below summarizes what have been proven about the possible Hausdorff and packing dimensions attained by transcendental entire functions.

\begin{figure}[!h]
	\label{DimTri}
	\includegraphics{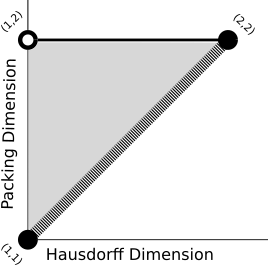}
	\caption{A graph showing the possible and attained Hausdorff and packing dimension pairs for transcendental entire functions. The point $(2,2)$ is attained by families of the exponential and sine functions. The upper segment is due to the work of Stallard, and the point $(1,1)$ is due to Bishop. Our contribution uses enlarged, dashed lines, to emphasize that have a dense set of dimensions attained very close to the diagonal.}
\end{figure}

We would like to point out how our construction differs from the constructions cited above. Since Stallard's examples belong to the Eremenko-Lyubich class, the packing dimension of those Julia sets must be $2$, even though the Hausdorff dimension can attain any value in $(1,2)$. The dynamical behavior of our examples is also much different; our functions have multiply connected Fatou components which do not occur in the Eremenko-Lyubich class. Stallard uses a family of functions defined via a Cauchy integral, whereas we use an infinite product construction similar to Bishop. Our example differs from Bishop's because not all of the Fatou components will be multiply connected. Instead of basing our construction off of a polynomial whose Julia set is a Cantor set, we base ours off of a polynomial whose Julia set is a quasicircle. It will follow that the entire function we construct also has Fatou components which get mapped onto an attracting basin whose boundary is a quasicircle. The multiply connected components of the Julia set will accumulate onto the boundary of these quasicircles, so that they are no longer close to being round everywhere. Bishop's techniques need to be extended to deal with both the copies of the attracting basin and the multiply connected Fatou components that are no longer round.

Finally, in his lecture at the 2019 Postgraduate Conference in Complex Dynamics, David Sixsmith asked what the possible dimensions of the \textit{bungee set} $BU(f)$ are for a transcendental entire function. The bungee set is definied in Section 14, and we refer the reader to \cite{OsbSix} for the basic properties of $BU(f)$. We are able to prove the following theorem.
\begin{theorem}
Given $s \in (1,2)$ and $\epsilon_0>0$, the function $f$ may be defined so that $\Hdim(BU(f)) \in (s-\epsilon_0, s+\epsilon_0)$. In other words, the possible dimensions of $BU(f)$ are dense in $(1,2)$. 
\end{theorem}

The author would like to thank Chris Bishop for suggesting this problem and for many useful conversations, suggestions, and for reading and offering detailed feedback on earlier drafts. David Sixsmith found many mistakes and typos and offered suggestions that greatly improved the exposition of this paper. The author would also like to thank Misha Lyubich, Lasse Rempe-Gillen, Gwyneth Stallard, and Phil Rippon for helpful discussions.

\section{Outline of the Proof}
We will construct a function $f: \bC \ra \bC$ depending on parameters $N \in \bN$, $R \in \bR$, and $c$ in the main cardiod of the Mandelbrot set. Define $f_0(z) = z^2 + c.$ The function $f$ will be an infinite product of the form
\begin{eqnarray}
\label{function}
f(z) = f_0^N(z) \cdot \prod_{k=1}^{\infty} \left(1 - \frac{1}{2}\left( \frac{z}{R_k} \right)^{n_k} \right) = f_0^N(z) (1 +\epsilon(z)).
\end{eqnarray}
Here, $n_k = 2^{N+k-1}$, and the sequence $\{R_k\}$ grows superexponentially and is defined inductively starting from a large initial parameter $R$. The choices are made so that near the origin, the infinite product can be made uniformly close to the constant function $1$. We write the infinite product as $(1+\epsilon(z))$ to emphasize this fact, where $\epsilon(z)$ is a holomorphic function uniformly close to the $0$ function in a large neighborhood of the origin.  

%Our construction mirrors Bishop's construction in \cite{CB1}, with a couple key changes, the most important being the iterated polynomial $z^2 + c$ has a quasicircle as its Julia set with exactly one attracting basin containing the origin, rather than a Cantor set. As previously mentioned, this will affect the computation of the packing dimension, since the quasicircle has dimension larger than $1$ and introduces geometry not present in Bishop's example. 

First, in section $5$, we will show that $f$ does indeed define an entire function. Given any $s \in (1,2)$, we will choose $c$ so that $\Hdim(J(f_0)) = s$. In a neighborhood of the origin, $f$ is a polynomial-like mapping which is close to $f_0^N$. We will construct a quaisconformal mapping with small dilatation mapping the Julia set of $f_0$ to the Julia set of the polynomial like mapping $f$.  It will follow that the Julia set of the entire function $f$ will have Hausdorff dimension bounded below by a value arbitrarily close to $s$ as well. 

In the next sections we prove several estimates on the growth of the sequence $\{R_k\}$, and decompose the plane into alternating annuli $A_n$ and $B_n$, where the modulus of $A_n$ is fixed and contains the circle $|z| = | R_n|$ and the modulus of $B_n$ increases as $n \ra \infty$. We will show that $f$ looks like a power function $z^m$ on $B_n$, that $f(B_n) \subset B_{n+1}$, and that if a point ever lands in $B_n$, it diverges locally uniformly to $\infty$ under $f$. Therefore, all the interesting dynamical behavior happens in the annuli $A_n$. We will show that $A_n \subset f(A_{n-1})$, and that all the zeros and critical points of $f$ and the Julia set are inside the $A_n$'s. To accomplish this, we will show (in a way that can be made precise) that $f$ is approximately equal to the $n$th term in the infinite product on $A_n$. 

From here, we will be able to prove that we can sort the Fatou components into two types depending on whether the component escapes to $\infty$ or remains bounded. The first type of Fatou component comes from the connected component containing the critical point $0$ of $f(z)$. This component is an attracting basin, and its inverse images form ``trapdoors'' in the sense that if $z$ is inside of one of the inverse images, $z$ will eventually land in the basin containing $0$ and remain there for the rest of its iteration. 

The second type are the components which are subsets of the escaping set. These components will be infinitely connected wandering domains, and the boundary of such components will be bounded by $C^1$ closed curves. These boundary curves will accumulate on the outermost boundary of each component. There is a distinguished sequence $\{\Omega_k\}_{k=-\infty}^{\infty}$ of these Fatou components which wind around the origin. We will split these components into two sub-categories. If $k \geq 1$, we will call $\Omega_k$ ``round" since the inner and outer boundary of $\Omega_k$ will be $C^1$ curves which are approximately circles. We will call $\Omega_k$ for $k \leq 0$ ``wiggly". The inner and outer boundary of wiggly components will be  $C^1$ curves that approximate the fractal boundary of the basin of attraction containing $0$. See Figure $2$. Since $\Omega_k$ is multiply connected, its complement can be split into components $\Omega_k^0$, the component contained in the origin, $\Omega_k^{\infty}$, the unbounded component, and countably many components $\Omega_k^a$ which are between the inner and outer boundary components. We will prove that some iterate of $f$ maps $\Omega_k^a$ conformally onto a domain bounded by the outer boundary of $\Omega_j$ for $j \geq 1$, and that the iterate of $f$ has small conformal distortion. It follows that $\Omega_k^a$ contains copies of $\Omega_l$ for $l \leq j$, and these copies look approximately like $\Omega_l$. 

\begin{figure}[!h]
\centering
\subfloat[Round Component]{\includegraphics[width=.4\textwidth]{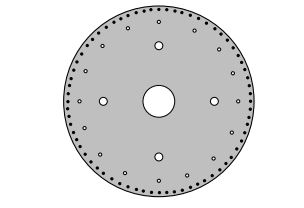}\label{fig:f1}}
\hfill
\subfloat[Wiggly Component]{\includegraphics[width=.4\textwidth]{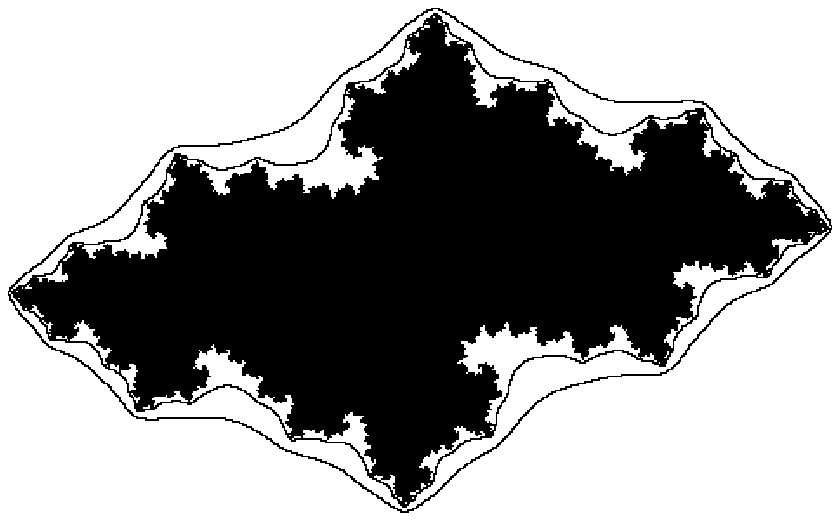}\label{fig:f2}}
\caption{On the left, we have a schematic of a round multiply connected Fatou component. This is the schematic for components $\Omega_k$ for $k \geq 1$. As $k \ra -\infty$, the components $\Omega_k$ trace the boundary of the fractal basin of attraction. On the right, we have the filled Julia set of $z^2 + c$ for $c = -.22 + .66i$ shaded in black. The boundaries of $\Omega_k$ for $k \leq 0$ look approximately like level lines for the Green's function of the complement of the filled Julia set, which we used to generate the image. Wiggly components are still multiply connected; we have omitted the holes in the picture on the right.}
\end{figure}

The Julia set of $f$ will contain the boundaries of each of these two types of components. This is not the entire Julia set though. Since $f$ has a multiply connected Fatou component, the work of Dominguez (\cite{Dom}) implies that the Julia set will also contain points that do not lie on the boundaries of either of these two types of components. These so-called ``buried" points in the Julia set either remain bounded, are in the bungee set, or escape, but they do not belong to the fast escaping set. In sections 10 through 13, we will perform a detailed analysis on the dimension of the set of these points. We will show that the dimension of this set is lower bounded by the dimension of the boundary of the basin of attraction containing $0$. While the dimension could possibly be larger than the dimension of the boundary of the basin of attraction, we show that we can make this difference in dimension arbitrarily small. 

To upper bound the packing dimension, we will follow \cite{CB1} and study the critical exponent of a Whitney decomposition of the complement of the Julia set of $f$ in a bounded region. Since the Julia set of $f$ will have zero area, it turns out that this critical exponent coincides with the packing dimension, and we will show that this exponent is at most the dimension of the buried points. The key idea in this part of the proof is to iterate small Fatou components to components of unit size where we can calculate the critical exponent of a Whitney decomposition directly. The tradeoff is that this conformal rescaling procedure results in various corrective factors that we must now control independently of which small Fatou component we chose. We do this with a combination of all the technical work done earlier in the paper.

\section{Hausdorff, Packing, and Minkowski Dimension}
Given a set $A \subset \bC$, we define its \textit{$\alpha$-Hausdorff measure} to be the quantity
$$H^{\alpha}(A) := \lim_{\delta \ra 0} H_{\delta}^{\alpha} (A) := \lim_{\delta \ra 0}\left( \inf \left \{\sum_{i=1}^{\infty} \diam(U_i)^{\alpha} \,: \, A \subset \cup_{i=1}^{\infty} U_i, \, \diam(U_i) < \delta\right \}\right).$$
The infimum is taken over all countable covers $\{U_i\}$ of A. One can check that if $H^{t}(A) < 0$, then $H^{s}(A) = 0$ for all $s > t$, and similarly, if $H^{t}(A) > 0$, then $H^s(A)  = \infty$ for all $s < t$. It follows that the \textit{Hausdorff dimension}
$$\Hdim (A) := \sup \{t: \, H^t(A) = \infty\} = \inf\{t: H^t(A) = 0\}$$
is uniquely defined. 

%The Hausdorff measure of a set is unchanged if we restrict the coverings to be open, closed, or even convex. It will change if we restrict our coverings to be by open balls, but only by a bounded amount, since every open set $U$ is contained in a ball around $z \in U$ of twice its diameter. It follows that the Hausdorff dimension does not change if we restrict our coverings to be by open balls, and similar arguments follow for other restricted coverings, such as squares or dyadic squares.

Given a compact set $K \subset \bC$, define $N(K,\epsilon)$ to be the minimal number of open balls of radius $\epsilon$ needed to cover $K$. Since $K$ is compact, this number exists and is finite. We define the \textit{upper Minkowski dimension} of $K$ to be 
$$\UMdim(K) = \limsup_{\epsilon \ra 0} \frac{\log(N(K,\epsilon))}{\log(1/\epsilon)} = \sup\{s \geq 0: \limsup_{\epsilon \ra 0} N(K,\epsilon)\epsilon^s = 0\}.$$
Equivalently, one may consider coverings of $K$ by squares of a fixed edge length, and since the diameters of squares and balls are comparable, this would not change the definitions above. For this reason, the upper Minkowski dimension is often called the \textit{upper box counting dimension}, although we will favor the former notation. 

In this paper, we will investigate the upper Minkowski and packing dimension of \textit{unbounded} Julia sets, so strictly speaking, the definition above does not make sense. We can instead consider the \textit{local upper Minkowski} dimension of the Julia set, which is the upper Minkowski dimension of the Julia set intersected with an open neighborhood of finite diameter. In \cite{RS}, Rippon and Stallard show that the local upper Minkowski dimension of the Julia set of an entire function is constant and coincides with its packing dimension (defined below), except perhaps in a neighborhood of $1$ point (a point with finite backward orbit; there is at most $1$ by the Picard theorem). Our example will not have an exceptional value of this kind, so their result further implies that the packing dimension and local upper Minkowski dimension are the same, no matter where we measure the local Minkowski dimension. In light of this, we will abuse notation and refer to the local upper Minkowski dimension of $\cJ(f)$ by $\UMdim(\cJ(f))$; the neighborhood we are using will always be made clear.

One reason to use the packing dimension is that the upper Minkowski dimension has poor measure theoretic qualities. In particular, one can check for the set $K = \{0\} \cup \{1/n:\, n= 1,2,\dots\}$ has Minkowski dimension $1/2$, but Hausdorff dimension $0$ since it is countable. This means that Minkowski dimension does not satisfy
$$\UMdim( \cup A_i) = \sup \UMdim(A_i).$$ 
To fix this issue, we define the \textit{packing dimension} of $K$ to be 
$$\Pdim(K) = \inf\left\{ \sup_i \{ \UMdim K_i :\, K = \cup K_i \} \right\}.$$
Here, the infimum is taken over all partitions of $K$ into countably many subsets $K_i$. The packing dimension can also be defined in terms of $\alpha$-packing measures, see Section $2.7$ in \cite{BP}. In practice, it is difficult to compute packing dimension using either of these definitions. We will instead use the strategy described below. 

For $n \in \bZ$, let $D_n$ denote the $nth$ generation of dyadic intervals 
$$I = [j2^{-n},(j+1)2^{-n}], \quad j \in \bZ.$$
We call $Q \subset \bC$ a \textit{dyadic cube} if it is the product of $2$ dyadic intervals inside the same $D_n$. Dyadic cubes satisfy the following simple but useful properties:
\begin{enumerate}
	\item The side length of a cube $Q$ is $l(Q) = 2^{-n}$, and its diameter is $\diam(Q) = \sqrt{d} \cdot l(Q)$. 
	\item Each dyadic cube is contained inside a unique cube parent cube $Q^{\uparrow}$ with the property that $\diam(Q^{\uparrow}) = 2 \diam(Q)$. 
	\item Given two dyadic cubes $Q_1$ and $Q_2$, either $Q_1$ or $Q_2$ have disjoint interiors, or one is contained in the other. 
\end{enumerate}
If $\Omega \subset \bC$ is open, each $x \in \Omega$ belongs to a dyadic cube $Q \subset \Omega$ with the additional property that $\diam(Q) \leq \dist(Q,\pa \Omega)$. By property (3), each point is contained in a cube of maximal possible diameter. It follows that the cubes $\{Q_j\}$ form a \textit{Whitney decomposition} of $\Omega$, that is, a collection of cubes that are have disjoint interior, cover $\Omega$, and satisfy the estimate
\begin{eqnarray}
\label{Whit}
\frac{1}{C} \dist(Q_j, \pa \Omega) \leq \diam(Q_j) \leq C \dist(Q_j, \pa \Omega),
\end{eqnarray}
for some $C >1$. More generally, we may consider Whitney decompositions of $\Omega$ with subsets that are not dyadic cubes, but are instead sets with disjoint interior whose closures cover $\Omega$ and satisfy the same estimate above. Whitney decompositions are also conformally invariant in the following sense: if $f: \Omega \ra f(\Omega)$ is a conformal map, and $Q$ is a cube in some Whitney decomposition of $\Omega$, then $f(\Omega)$ is covered by at $M$ cubes in a Whitney decomposition of $f(\Omega)$, where $M$ only depends on the constants in $(\ref{Whit})$ and not on the conformal mapping (see \cite{Garnett}, p. 21).

Whitney decompositions allow us to define for any compact $K \subset \bC$ the \textit{critical exponent}
$$\alpha = \alpha(K) = \inf\{ \alpha: \, \sum \diam(Q)^{\alpha} < \infty\}$$
where the sum is taken over all cubes in a Whitney decomposition $\bC \setminus K$ within some bounded distance of $K$. The critical exponent is well-defined; it does not depend on which Whitney decomposition we choose for the complement of $K$. The key point is that given two Whitney decompositions of a domain $\Omega$, and a cube inside one of the decompositions, it can be covered by a finite number of cubes in the other collection, and this finite number depends only on the constants in (\ref{Whit}) and the dimension $d$. The following is Lemma $2.6.1$ in \cite{BP}, which relates the critical exponent of the complement of a set to its upper Minkowski dimension.
\begin{lemma}
	For any compact $K \subset \bC$, $\alpha(K ) \leq \UMdim(K)$. If $K$ has zero Lebesgue measure then we have equality.
	\label{VE}
\end{lemma}
To summarize this section, for our application, we have 
$$\Hdim(\cJ(f)) \leq \Pdim(\cJ(f)) = \UMdim(\cJ(f)).$$
Therefore, to obtain Julia sets with packing dimension in $(1,2)$, our approach will be to show that given $s \in (1,2)$ and $\delta > 0$, we can define $f$ so that $\Hdim(\cJ(f)) >s - \delta$ and $\UMdim(\cJ(f)) < s + \delta$, for which we will use the lemma above. A complete discussion of these various dimensions, including proofs, can be found in \cite{BP}.

\section{Quasiconformal and Polynomial-Like Mappings}
Douady and Hubbard introduced polynomial-like mappings in \cite{Hub}. In this section we outline and review the properties we will need in this paper.

First, recall that a continuous mapping $f: U \ra V$ between two topological spaces is called \textit{proper} if the inverse image of every compact $K \subset V$ is compact.  A \textit{degree $d$ polynomial-like map} is a triple $(f,U,V)$, where $f: U \ra V$ is a proper holomorphic mapping of degree $d$, and $U$ and $V$ are homeomorphic to disks with $U$ is relatively compact in $V$. We define the \textit{filled Julia set} of $f$ by 
$$K_f := \bigcap_{n\geq0} f^{-n}(U).$$
This is precisely the set of points that remain in $U$ for all iterates of $f$. The Julia set of $f$ is defined to be the boundary $\pa K_f$, and we denote it by $J_f$.

Many of the classical propositions in the dynamics of polynomials are true for polynomial-like mappings. For example, suppose that $B \subset K_f$ is the immediate basin of attraction for some attracting fixed point $p \in B$. The same proof as in the case of rational functions that $B$ contains a critical point of $f$ applies equally well to the polynomial-like case. Hyperbolicity also makes sense in the setting of polynomial-like maps. We say $f: U \ra V$ is \textit{hyperbolic} if there exists a Riemannian metric on $J_f$ and $a > 1$ so that if $z \in J_f$ we have $||D_zf(v)||_{f(z)} \geq a ||v||_z$ where $v \in T_z J_f$ is a tangent vector. With the same proof as the case of rational mappings, this definition is equivalent to the definitions that $|(f^N)'| > 1$ on $J$ for some $N$, or that every critical point is attracted to an attracting fixed point or cycle. In our applications, the polynomial-like map $f$ will come as the restriction of an entire function, and  it will be important that we distinguish between $f$ being hyperbolic as a polynomial-like map, versus $f$ being hyperbolic as a transcendental entire function (which our example cannot be, since it has an unbounded set of critical points. See \cite{LRG}.) 

The usefulness of polynomial-like mappings comes from the straightening lemma (Theorem 1, p. 296 in \cite{Hub}), which gives a quasiconformal conjugacy between the polynomial-like map and some polynomial of the same degree. To lower bound the Hausdorff dimension of the Julia set of our examples, we will construct a quasiconformal conjugacy in a way similar to the straightening lemma. If $U$ and $V$ are planar domains, we call an orientation preserving homeomorphism $\varphi: U \ra V$ \textit{$K$-quasiconformal} if $\varphi$ has locally square integrable distributional derivatives which satisfy
$$|\varphi_{\bar z}(z)| \leq k |\varphi_z(z)|$$
for all $z \in U$ for $k = (K-1)/(K+1) < 1.$ Given a quasiconformal mapping, we define its \textit{dilatation} 
$$\mu(z) = \frac{\varphi_{\bar z}(z)}{\varphi_z(z)}.$$
The definition says that the dilatation is bounded above by some number less than $1$. We will need the following fact about quasiconformal mappings. The proof can be found in Section IV.5.6 of \cite{LV}.

\begin{lemma}[The Good Approximation Lemma]
	\label{GAL}
	Suppose that $\varphi_n: U \ra V$ is a sequence of quasiconformal mappings. Suppose that $\varphi_n \ra \varphi$ uniformly on compact sets, and the corresponding dilatation $\mu_n$ of $\varphi_n$ converges pointwise almost everywhere to some limit almost everywhere. Then $\varphi$ is quasiconformal with dilatation $\mu = \lim \mu_n$.
\end{lemma}

\section{Defining the Function}
In this section, we specify the parameters defining $f$ and show that it is an entire function. Along the way we will prove some basic estimates regarding the parameters $\{R_k\}$ and $n_k$ defining $f$ that will be useful later in the paper.  

Recall that the main cardioid of the Mandelbrot set is the region consisting of all parameters $c = \mu/2(1-\mu/2)$, where $\mu \in \bD$. If $c$ is a parameter in the main cardioid, the Julia set of $z^2 + c$ is a quasicircle with an attracting fixed point in its interior.  As mentioned in the introduction, the results of Shishikura and Sullivan imply that for each $s \in (1,2)$, we may choose $c$ in the main cardioid so that $\Hdim(\cJ(F_0)) = \Pdim(\cJ(F_0(z)) = s$ (see \cite{Shish} p.232 and \cite{DS1} p.742, along with Theorem 7.6.7 in \cite{Urb}).

Having chosen such a $c$, recall that we defined $f_0(z) = z^2 + c$, and $f_0^N(z)$ denotes the $N$th iterate of $f_0$. Since $f_0^N$ is a degree $2^N$ monic polynomial there exists some $R> 0$ so that if $|z| \geq R$ we have
\begin{eqnarray}
\frac{1}{2} \leq \left | \frac{f^N_0(z)}{z^{2^N}} \right | \leq 2.
\label{one}
\end{eqnarray}
We will always assume $R$ is big enough so that $(\ref{one})$ holds.

Next given some integer $N > 0$ define a sequence of integers for $k =0,1,2 \dots$
$$n_k := 2^{N+k-1}.$$
Note that when $k \neq 0$, $n_k \geq 2^N$ and for all $k$ we have $2n_k = n_{k+1}$.
Given the $R$ above, define 
$$R_1 = 2 R.$$
We will construct our infinite product as a sequence of partial products inductively as follows. Given $R$ as above we can define
$$F_1(z) :=  \left(1 - \frac{1}{2}\left( \frac{z}{R_1} \right)^{n_1} \right),$$
$$f_1(z) := f^N_0(z) \cdot F_1(z),$$
$$R_2 :=  M(f_1, 2R_1) = \max\{ |f_1(z)| \,:\, |z| = 2R_1\}.$$
Next, assume that $f_{k-1}$, $F_{k-1}$ and $R_k$ have all been defined. From there, we define
$$F_k(z) := \left(1 - \frac{1}{2}\left( \frac{z}{R_k} \right)^{n_k} \right),$$
$$f_k(z) := f_0^N(z) \prod_{j=1}^{k} F_j(z),$$
$$R_{k+1} := M(f_k, 2R_k) = \max\{ |f_k(z)| \,:\, |z| = 2R_k\}.$$
With these starting parameters, we want to begin by looking at functions of the form
\begin{eqnarray}
f(z) = \lim_{k \ra \infty} f_k(z) = \lim_{k \ra \infty} f_0^N(z) \prod_{j=1}^k F_j(z).
\label{Def}
\end{eqnarray}
Viewing this as a formal infinite product, our first step will be to show that $f$ is indeed an entire function on $\bC$. 

\begin{lemma}[The Growth Rate of $n_k$]
	For all $k =1,2,\dots$, we have
	\begin{enumerate}
		\item $n_k = 2n_{k-1}$, and $n_k \geq 2^N$.
		\item $2^N+ \sum_{j=1}^{k} n_j =  n_{k+1}$.
	\end{enumerate}
	\label{nk}
\end{lemma}
\begin{proof}
	The first claim immediate from the definitions. For the second claim we compute
	\begin{eqnarray*}
		\sum_{j=1}^{k} n_j &=& (2^N + \dots + 2^{N+ k-1} ) 
		= 2^{N-1}(2 +\dots + 2^{k} ) \\
		&=& 2^{N-1}(2^{k+1} - 2) 
		= 2^{N+k} - 2^N \\
		&=& n_{k+1} - 2^N.
	\end{eqnarray*}
	This is just a rearranged version of $(2)$.
\end{proof}

\begin{corollary}
	For all $k \geq 1$, $\deg(f_k) = 2\deg(F_k)$
	\label{nk2}
\end{corollary}
\begin{proof}
	Apply Lemma \ref{nk} to the equality $\deg(f_k) = 2^N + \sum_{j=1}^{k} n_j$.
\end{proof}

We can deduce the following growth rate estimates for $R_k$.
\begin{lemma}[The Growth Rate of $R_k$]
	If $k \geq 1$, $R$ satisfies $(\ref{one})$, and if $N \geq 10$, we have 
	$$R_{k+1} \geq 2^{n_k} R_k^{2^{N-1}+n_{k-1}} \geq 2^{N} R_k^{2^N}.$$
	\label{Rk}
\end{lemma}
\begin{proof}
	$R$ is big enough so that $(\ref{one})$ holds, so that
	\begin{eqnarray*}
	R_2 &:=& \max_{|z|=2R_1} |f_0^N(z)| \cdot \left|\left(1-\frac{1}{2}\frac{z^{n_1}}{R_1^{n_1}} \right)\right| \\
	&\geq& \frac{1}{2} |2R_1|^{2^N} \cdot \max_{|z|=2R_1} \left|\left(1-\frac{1}{2}\frac{z^{n_1}}{R_1^{n_1}} \right)\right| \\
	&\geq& 2^{2^N-1} R_1^{2^N} \cdot (2^{n_1-1}-1) \\
	&\geq& 2^{2^N} R_1^{2^N} = 2^{n_1} R_1^{2^{N-1} +n_0} \geq 4 R_1^2.
	\end{eqnarray*}
	This is the base case for an induction. Suppose that for some $k$, we have 
	$$R_j \geq 2^{n_{j-1}} R_j^{2^{N-1}+ n_{j-2}} \geq 2^{2^N}  R_j^{2^N} \geq 4 R_j^2$$
	for all $2 \leq j \leq k$. This induction hypothesis implies that $\sqrt{R_{k}} \geq R_j$ for all $j \leq k-1$. Therefore,
	\begin{eqnarray*}
	R_{k+1} &:=& \max_{|z| = 2R_k} |f_0^N(z)| \cdot \prod_{j=1}^k \left| \left(1- \frac{1}{2}\frac{z^{n_j}}{R_j^{n_j}} \right)\right|\\
	&\geq& 2^{2^N-1}R_k^{2^N} \prod_{j=1}^k \left|2^{n_j-1}\frac{R_k^{n_j}}{R_j^{n_j}}-1 \right| \\
	&\geq&  2^{2^N-1}R_k^{2^N} (2^{n_k-1} -1) \prod_{j=1}^{k-1} \left|2^{n_j-1} R_k^{n_{j-1}} - 1\right| \\
	&\geq& 2^{2^N-2(k+1)+\sum_{j=1}^k n_j } R_k^{2^N+\sum_{j=1}^{k-1} n_{j-1}} \\
	&\geq& 2^{2^N -2(k+1) + n_{k+1}} R_k^{2^N + n_{k-1}} \\
	&\geq& 2^{2^N + n_k} R_k^{2^N + n_{k-1}}. 
	\end{eqnarray*}
To get the last inequality, we used the fact that $n_{k+1} = 2^{N+k} \geq 2(k+1)$ when $N \geq 10$ for all $k \geq 1$. Therefore $R_{k+1}$ satisfies the inequality in the lemma.  
\end{proof}

For the rest of the paper, we will always assume that $N \geq 10$, so that the conclusion of Lemma $\ref{Rk}$ is always valid. The lemma above also contains the following simpler inequalities that will often be sufficient for our purposes. We note them below.
\begin{corollary}[Other Useful Inequalities]
	For $k \geq 1$ we have 
	$$R_{k+1} \geq 4 R_k^{2}.$$
	For $k \geq 1$ we have:
	$$R_{k+1} \geq (2R)^{2^{kN}}.$$
	\label{Rk2}
\end{corollary}
\begin{proof}
	The first inequality is a weaker version than the one found in Lemma $\ref{Rk}$. The second inequality follows from Lemma $\ref{Rk}$ by induction and the fact that $R_1 = 2R$. Indeed, the base case is clear by Lemma $\ref{Rk}$, and if the claim is true for $R_k$, then 
	$$R_{k+1} \geq (R_k)^{2^N} \geq (2R)^{2^N \cdot 2^{{N(k-1)}}} \geq (2R)^{2^{N(k-1)+N}}.$$
\end{proof}
Now we show that the infinite product we are interested in converges on $\bC$. 
\begin{corollary}
	The infinite product
	$$f(z) = \lim_{k \ra \infty} f_k(z) = \prod_{k=0}^{\infty} F_k(z)$$
	converges uniformly on compact subsets of $\bC$. In particular, $f(z)$ is a transcendental entire function. 
\end{corollary}
\begin{proof}
	To check that the infinite product converges, we check that the associated sum
	$$\sum_{k=0}^{\infty} | 1 - F_k(z) |$$
	converges uniformly on compact sets. It suffices to show that $f$ converges uniformly on every closed ball $\{|z| \leq s\}$. If we choose $j$ so that $2s < R_j$,  then we know that for all $k > j$ 
	$$|1-F_k(z)| =\frac{1}{2}\left| \frac{z}{R_k} \right|^{n_k}  \leq \frac{s^{n_k}}{R_k^{n_k}} \leq \left(\frac{R_j}{2R_k}\right)^{n_k} =  O(2^{-n_k}) = O(2^{-2^{k}}).$$
	In this way, the series above is summable, so the series converges uniformly on  $\{|z| \leq s\}$. So $f$ defines an entire function, and it is certainly not a polynomial.
\end{proof}

We conclude this section by recording two useful lemmas regarding the growth rate of $R$. The first lemma will help us study $f$ near $|z| = R_k$ in section $7$. The proof follows from Theorem \ref{Rk}, and we refer the reader to sections $6$ and $8$ of \cite{CB1} for the details of the proof.

\begin{lemma}
	Suppose that $\{R_k\}$ has been defined as in this section, and $m \geq 1$. Then 
	\begin{eqnarray}
	\prod_{j=1}^{k-1} \left(1 + \left(\frac{R_j}{R_k} \right)^m \right) = 1 + O(R_k^{-1/2}).
	\end{eqnarray}
	\begin{eqnarray}
	\prod_{j=k+1}^{\infty}\left(1 + \frac{R_k}{R_j} \right)  = 1 + O(R_k^{-1}).
	\end{eqnarray}
	Finally, if $|z| \leq 4R_k$, we have
	\begin{eqnarray}
	\prod_{j=k+1}^{\infty} F_j(z) = 1 + O(R_k^{-1}).
	\end{eqnarray}
	\label{bigoh}
\end{lemma}

This second lemma is used in the proof of Lemma 11.1
\begin{lemma}
	Let $N_k = n_1 \cdots n_k$, and let $\alpha >0$. Then for any $R >1$, 
	$$\sum_{k=1}^{\infty} 2^k N_k R_k^{-\alpha} < \infty$$
 	Moreover, by choosing $R$ sufficiently large, the sum can be made arbitrarily small.
	\label{giantsum}
\end{lemma}

\section{The Hausdorff Dimension Changes by a Small Amount}
Recall that we denote $f_0(z) = z^2 +c$. Since the Julia sets of $f_0$ and $f_0^N$ are the same, they have the same Hausdorff dimension. We can also view the function $f$ as $f_0^N$ perturbed by the infinite product we constructed in the previous section. The parameters $R$ have been chosen so that in a neighborhood around the origin the infinite product defining $f$ is uniformly close to $1$. To emphasize this, we write
$$f(z) = f_0^N(z) \cdot (1+ \epsilon(z))$$
where $1 + \epsilon(z)$ converges locally uniformly to the constant function $1$ as $R \ra \infty$.

In this section, we will view $f$ and $f_0^N$ as polynomial-like mappings. Recall that we denote the filled Julia set and Julia set of a polynomial-like mapping $f$ by $K_f$ and $J_f$, respectively. Since $f$ is an entire function, $J_f$ is a subset of the Julia set of $f$ viewed as an entire function. Since $f_0^N$ is a polynomial, the filled Julia set $K_{f_0^N}$ and the Julia set $J_{f_0^N}$ coincide with its usual filled Julia set and Julia set. We will denote the basin of attraction of $f_0^N$ by $B_{f^N_0}$, so that $K_{f^N_0} = J_{f_0^N} \cup B_{f^N_0}$. We will use similar notation for $f$. Recall that we chose $c$ so that $\Hdim(J_{f_0^N}) = s$. Our goal is to show that since $f$ and $f_0^N$ differ only by a perturbation close to the identity, $\Hdim(J_f) = t >1$ where $|s-t|$ is as small as we would like. It will follow that the Julia set of the entire function $f$, $\cJ(f)$, is at least $t$.  

The following lemma is obvious.
\begin{lemma}
	\label{polynomiallike}
Fix the parameter $N$ and fix some $r > 10$. Then for all $R$ sufficiently large, $V = f(B(0,r))$ is compactly contained in $B(0,R/4)$, and 
$$f: B(0,r) \ra V$$
is a degree $2^N$ polynomial like mapping.  
\end{lemma}
In Lemma $\ref{polynomiallike}$, larger choices of both $N$ and $r$ will require larger choices of $R$. We will make a definitive choice for $N$ at the end of Section $8$. Once this is done, we can make a definitive choice for $r$, and then we will always assume that $R$ is large enough so that the conclusions of Lemma $\ref{polynomiallike}$ hold. 

Next observe that
\begin{eqnarray}
\label{CloseDerivative}
|f'(z)| \geq |(f_0^N)'(z)| (1 + \epsilon(z))| - |\epsilon'(z)| |f_0^N(z)|.
\end{eqnarray}
On $B(0,r)$, we may choose $R$ sufficiently large so that $(1+\epsilon(z))$ is uniformly close to the constant function $1$. It follows that the first term on the right hand side of $(\ref{CloseDerivative})$ can be made arbitrarily close to $|(f_0^N)'(z)|$ by choosing $R$ large. $|f_0^N(z)|$ is bounded on $B(0,r)$, and the Cauchy estimates imply that if $z \in B(0,r)$, then
$$ |\epsilon'(z)| \leq \frac{\max_{w \in B(z,r)}|\epsilon(z)|}{r}.$$
Therefore, 
\begin{eqnarray}
\label{Secondterm}
\max_{z \in B(0,r)} |\epsilon'(z)| \leq \frac{\max_{w \in B(0,2r)}|\epsilon(z)|}{10}.
\end{eqnarray}
By perhaps making a larger choice of $R$, we can make $(\ref{Secondterm})$ as small as we would like. We summarize this discussion below.

\begin{lemma}
	\label{SimHyper}
Fix $N$ and $r > 10$. Let $\delta > 0$ be given. Then for all $R$ sufficiently large, 
$$\sup_{z \in B(0,r)}|f'(z) - (f^N_0)'(z)| < \delta.$$
\end{lemma}

We now prove the main theorem of this section.

\begin{theorem}
	Let $s = \Hdim(\cJ(f_0))$ and $\epsilon > 0$ be given. Then $R$ and $r$ may be chosen so that the Julia set of the polynomial-like mapping $f$ is a quasicircle such that $|\Hdim(J_f) - s| < \epsilon.$ 
\end{theorem}
\begin{proof}
	 $f_0^N$ is a hyperbolic polynomial since $c$ is in the main cardioid. It follows that there exists some topological annulus $A$ containing the Julia set of $f_0^N$ so that $|(f_0^N)'(z)| \geq \alpha > 1$ on $A$. By Lemma $\ref{SimHyper}$, we may arrange for $|f'(z)| \geq \alpha' > 1$ on $A$ as well. Since the dynamics are expanding, $A$ is relatively compact in $f^N_0(A)$ and $f(A)$, which are also topological annuli containing $A$.
	
	Since $f$ is uniformly close to $f^N_{0}$, the boundary of $f(A)$ and $f^N_0(A)$ are close to each other in the following sense. Let $z \in f(A)$ and let $D$ be a small disk around $z$. Then the outward and inward pointing normal vector $\textbf{n}_z$ based at $z$ makes small angles with any other inward or outward pointing normal vector $\textbf{n}_w$ based at $w$ in both $D$ and the boundary of $f^N_0(A)$. It follows that there exists a quasiconformal mapping
	$$\varphi_0: \overline{f_0(A) \setminus A} \ra \overline{f(A) \setminus A}.$$
	Moreover, $\varphi_0$ can be chosen to have the following properties:
	\begin{enumerate}
		\item $\varphi_0 = \mbox{id}$ on $\pa A$.
		\item $\varphi_0$ is $(1+\delta)$-quasiconformal, where $\delta \ra 0$ and $R \ra \infty$. 
		\item $f(\varphi_0(z)) = \varphi_0(f_{0}^N(z))$ on $\pa A$.
	\end{enumerate} 
	
	We now set up the following notation. We call $A_n = (f^N_0)^{-n}(A)$, and we call $B_n = f^{-n}(A)$. Thus $A_n$ and $B_n$ each form nested sequences of topological annuli. since $f_0^N$ and $f$ are both expanding on $A$, we know that the widths of $A_n$ and $B_n$ are both shrinking uniformly.
	
	Next, notice that $U_1 = A \setminus A_1$ and $V_1 = A \setminus B_1$ are each the disjoint union of topological annuli. We call the components the outer and inner annuli, corresponding to which component has larger diameter. We denote these annuli by $U^1_o$ and $V^1_o$ for the outer annuli, and $U^1_i$ and $V^1_i$ for the inner annuli. We can continue this procedure inductively, obtaining a sequence annuli $U_n = A_n \setminus A_{n+1} = U^n_o \cup U^n_i$ and $V_n = B_n \setminus B_{n+1} = V^n_o \cup V^n_i$. These annuli have the property that the outer boundary of $U^n_o$ coincides with the inner boundary of $U^{n-1}_o$ for all $n$, and the inner annuli have the property that the outer boundary of $U^n_i$ coincides with the inner boundary of $U^{n+1}_i$. Similar statements are true for $V^n_o$ and $V^n_i$. Our goal is, for each $n$, to construct a quasiconformal $\varphi_n: f_0(A) \ra f(A)$ satisfying 
	\begin{enumerate}
		\item $\varphi_n$ is conformal from $A_n$ to $B_n$,
		\item $f(\varphi_n(z)) = \varphi_n(f_0^N(z))$ on the components $U^k_i$, $U^k_0$, $V^k_i$, $V^k_0$ for $k =1,\dots, n$. 
		\item $\varphi_{n-1} = \varphi_n$ on the components $U^k_i$, $U^k_0$, $V^k_i$, $V^k_0$ for $k =1,\dots, n-1$. 
		\item $\varphi_n$ is $(1+\delta)$-quasiconformal, with the same $\delta$ as above. 
	\end{enumerate}
	
	Let us first construct $\varphi_1$. Note that $\varphi_0 \circ f^N_0: U_1 \ra f(A) \setminus A$ and $f: V_1 \ra f(A) \setminus A$ are each $2^N$ to $1$ continuous mappings, hence there exists a lift $\varphi_1: U_1 \ra V_1$ that satisfies 
	$$ \varphi_0 \circ f^N_0 = f \circ \varphi_1.$$
	Since $f$ and $f_0^N$ are are $2^N$ to $1$ and locally conformal mappings, $\varphi_1$ is $(1+\delta)$-quasiconformal. Next, since $\mbox{id} \circ f^N_0 : A_1 \ra A$ and $f: B_1 \ra A$ are $2^N$ to $1$, there is a lift $\Phi: A_1 \ra B_1$ so that 
	$$\mbox{id} \circ f^N_0 = f \circ \Phi.$$
	Note that since each of $f$ and $f^N_0$ are locally conformal and each $2^N-1$, $\Phi$ must be a conformal mapping. We define
	$$\varphi_1 = \begin{cases}
	\varphi_0(z) & z \in f_0(A) \setminus A, \\
	\varphi_1(z) & z \in U_1, \\
	\Phi(z) & z \in A_1. 
	\end{cases}$$
	We should check that $\varphi_1$ is continuous. Indeed, if $z$ is on the boundary of $U_1$ and $f_0(A) \setminus A$ then 
	$f(\varphi_1(z)) = \varphi_0(f_0(z)) = f(\varphi_0(z))$ by the construction of $\varphi_0$. So we choose the lift $\varphi_1$ (there are $2^N$ possible choices) so that $\varphi_1 = \varphi_0$ on this common boundary. If $z$ is on the boundary of $U_1$ and $A_1$, similarly we argue that 
	$$f( \Phi(z)) = f_0(z) = \varphi_0(f_0(z)) = f(\varphi_1(z)).$$
	So we choose the lift so that $\Phi(z) = \varphi_1(z)$ on their common boundary of definition. 
	
	We may continue this procedure inductively, obtaining lifts $\varphi_n: U_n \ra V_n$ and $\Phi_n: A_n \ra B_n$ that satisfy
	$$f(\varphi_n(z)) = \varphi_{n-1}(f_0(z)) \,\,\,\,\,\,\, \mbox{on $U_n$, and,}$$
	$$f(\Phi_n(z)) = \Phi_{n-1}(f_0(z)) \,\,\,\,\,\mbox{on $A_n$.}$$
	Just as before, we may choose the lifts so that these maps agree on their common boundaries. $\Phi_n$ is conformal on $A_n$ and $\varphi_n$ is $(1+\delta)$-quasiconformal on $U_n$. We define
	$$\varphi_n = \begin{cases}
	\varphi_{n-1}(z) & z \in \left(f_0(A) \setminus A\right) \cup \bigcup_{i=1}^{n-1} U_i, \\
	\varphi_n(z) & z \in U_n, \\
	\Phi_n(z) & z \in A_n. 
	\end{cases}$$
	By construction, $\varphi_n$ evidently satisfies the four properties outlined above.
	
	The result is a sequence of $(1+\delta)$-quasiconformal mappings $\varphi_n: f_0(A) \ra f(A)$. The dilatation $\mu_n$ converges almost everywhere, since it is eventually constant on each $U_n$. The $\varphi_n$ family converges uniformly on compact sets as well. Indeed, since $\varphi_n$ conjugates the dynamics, it sends the fixed points of $f$ to fixed points of $f_0^N$. By perhaps taking a subsequence, since there are only finitely many repelling fixed points for $f_0^N$, we can assume that $\varphi_n(z) = w$ for some repelling fixed point $z$ of $f^N_0$ and $w$ of $f$ for all $n$. The sequence $\{\varphi_n\}$ is therefore sequentially compact, every sequence has a convergent subsequence. We obtain a limit $\varphi(z)$ which conjugates the dynamics and maps the Julia set of the polynomial $f_0^N$ onto the Julia set of the polynomial-like map $f$. By the Lemma $\ref{GAL}$, the limit $\varphi$ is $(1+\delta)$-quasiconformal. Since $\varphi$ is $\alpha$-H{\"o}lder with $\alpha$ close to one, the dimension of $J_f$ changes only by a small amount (see \cite{Ahlfors}, p.30). 
\end{proof}
The following corollaries are immediate from the above proof.
\begin{corollary}
	The Julia set of the polynomial-like mapping $f$ is a quasicircle with dimension in $(1,2)$ arbitrarily close to $\dim(\cJ(z^2+c))$.
\end{corollary}
\begin{proof}
	The Julia set of $f$ is the quasiconformal image of a quasicircle. The rest is just a restatement of Theorem $6.2$.
\end{proof}

\begin{corollary}
	$f$ is a hyperbolic polynomial-like mapping. 
\end{corollary}
\begin{proof}
	We showed $f$ was expanding on $\cJ(f)$ in the proof above.
\end{proof}

To conclude this section, we review some notation we will use for the rest of the paper. We now know $f: B(0,r) \ra V$ viewed as a polynomial-like mapping has a quasicircle Julia set and one attracting basin. We will let $K_f$ denote the filled Julia set of the polynomial-like mapping $f$, and write $K_f = J_f \cup B_f$, where $J_f$ is the Julia set and $B_f$ is the attracting basin.

\section{The Local Behavior of $f$}
We now move on to analyzing the function $f$ away from the origin. The purpose of this section is to show that $f$ behaves like simpler functions (the functions $F_j$ defined in Section $5$ and simple power functins) on suitably defined regions of $\bC$.  Recall that
$$f(z) := f_0^N(z) \cdot \prod_{j=1}^{\infty} F_j(z).$$
To be more specific, we will  show, quantitatively, that we can decompose $\bC \setminus B(0,R_1/4)$ into regions where $f$ looks approximately like $F_j$. The observations and estimates here are vital for understanding to precise dynamical behavior of $f$. 

We define 
$$H_m(z) = z^{m}(2-z^m).$$
A description of the conformal mapping behavior of $H_m$ can be found in Section $9$ of \cite{CB1}. For our purposes, we will need to consider the connected components of $\bC \setminus \{|H_m(z)| = 1\}$. This set has $m+2$ connected components, one unbounded, one containing the origin, and $m$ petals which we denote by $\Omega^p_m$. $H_m: \Omega^p_m \ra \bD$ is a conformal mapping, and $\diam(\Omega_m^p) = O(1/m)$. See Figure 3.

\begin{figure}[!h]
	\centering{\includegraphics[width=\textwidth]{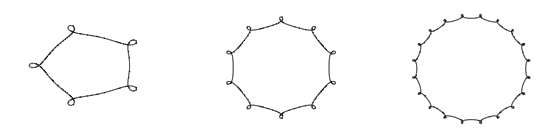}}
	\caption{A an illustration of the level sets of $\{|H_m(z)| = 1\}$ for $m = 5,10$ and $20$. There are $m$ petals where $H_m$ is a conformal mapping to the disk, and as $m$ grows, the diameter of the petals shrinks. All the points on $\{|H_m(z)| = 1\}$ are distance $O(1/m)$ from the unit circle $|z| =1$.}
\end{figure} 
 
Next, we decompose $\bC \setminus B(0,  R_1/4)$ into annuli as follows. 
$$A_k := \left\{ z: \frac{1}{4} R_k \leq |z| \leq 4 R_k \right \},
\,\,\,\,\,\,\,\,\,\,\,\,\,\,\, B_k := \left\{ z: 4  R_k \leq |z| \leq \frac{1}{4}  R_{k+1} \right \},$$
$$V_k := \left\{ z: \frac{3}{2} R_k \leq |z| \leq \frac{5}{2}R_k \right \},\,\,\,\,\,\,\,\,\,\ U_k := \left\{ z: \frac{5}{4}  R_k \leq |z| \leq 3 R_k \right \}.$$
Note that $V_k$ is compactly contained inside of $U_k$. See Figure 4.
\begin{figure}[!h]
\centerline{ \includegraphics[height=3in,width=3in]{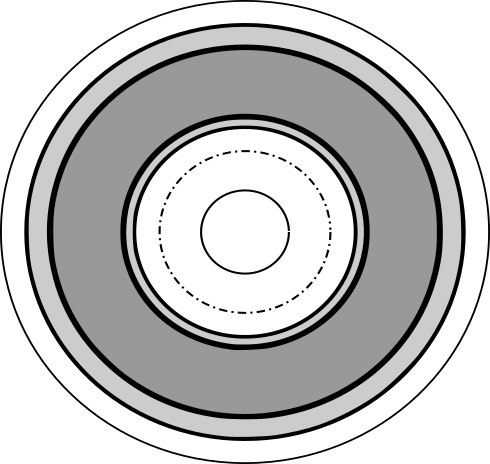} }
\caption{A Schematic for $A_k$, $k \geq 1$. The innermost circle and outermost circle form the boundary of $A_k$. The dashed line is the $z$ with $|z| = R_k$. The lightly shaded region is $U_k$, and the darker region is $V_k$, which is compactly contained in $U_k$. In the upcoming sections, we will see that the Julia set if $f$ is contained near the circle $|z| = R_k$ and in $V_k$.}
\end{figure}

\begin{lemma}
	With $H_m$ defined above, we have 
	$$F_{k} (z) = \frac{1}{2} \left(\frac{ R_k}{z}\right)^{n_k} H_{n_k}\left (\frac{z}{R_k} \right).$$
\end{lemma}
\begin{proof}
	This is a simple computation: 
	\begin{eqnarray*}
	\frac{1}{2} \left(\frac{R_k}{z}\right)^{n_k} H_{n_k}\left (\frac{z}{R_k} \right) &=&  \frac{1}{2} \left(\frac{R_k}{z}\right)^{n_k} \left (\frac{z}{R_k} \right)^{n_k} \left (2- \left (\frac{z}{R_k} \right)^{n_k} \right) \\
	&=& \frac{1}{2} \left (2- \left (\frac{z}{R_k} \right)^{n_k} \right) \\
	&=& F_k(z). 
	\end{eqnarray*}
\end{proof}

\begin{lemma}
	If $z \in A_k$, there is a constant $C_k$ so that 
	$$f (z) = C_k H_{n_k}\left( \frac{z}{ R_k} \right) (1 + O( R_k^{-1})).$$
	For $k \geq 2$, the constant $C_k$ is given by the formula
	$$C_k = (-1)^{k-1}2^{-k}  R_k^{n_k} \prod_{j=1}^{k-1}  R_j^{-n_j}.$$
	For $k =1$ the constant is given by
	$$C_1 = \frac{1}{2} R_1^{n_1}.$$
	\label{A_k}
\end{lemma}
\begin{proof}
	Write 
	$$f(z) = f_0^N(z) \cdot \left( \prod_{j = 1}^{k-1} F_j(z) \right) \cdot F_k(z) \cdot \left( \prod_{j=k+1}^{\infty} F_j(z) \right)$$
	The tail end of the infinite product is easy to estimate; it satisfies the hypotheses of Lemma $5.3$, so we obtain
	$$  \prod_{j=k+1}^{\infty} F_j(z)  = 1 + O(  R_k^{-1}).$$
	Next we use Lemma part $(2)$ of \ref{nk} to compute
	\begin{eqnarray*}
		f_0^N(z) \cdot \left( \prod_{j = 1}^{k-1} F_j(z) \right) \cdot F_k(z)  &=&\left(  z^{-2^N} f_0^N(z)\right)  \left( \prod_{j = 1}^{k-1}z^{-n_j} F_j(z) \right) \left( z^{n_k} F_k(z)\right) \\
		&=& I \cdot II \cdot III. 
	\end{eqnarray*}
	We show how to estimate each term. Note first that $z \in A_k$, $|z| > R$, so that by $(\ref{one})$,
	$$I = z^{-2^N}f_0^N(z) = 1 + O(R_k^{-1}).$$
	Similarly, we can check that for $j=1,\dots,k-1$ that
	$$z^{-n_j} F_j( z) = z^{-n_j} \left (1- \frac{1}{2}\left (\frac{z}{ R_k} \right)^{n_j} \right) = \frac{-1}{2} R_j^{-n_j} \left( 1 - 2\left ( \frac{ R_j}{z} \right)^{n_j} \right).$$
	Since $z \in A_k$, $|z|$ is comparable to $ R_k$. Therefore 
	$$z^{-n_j} F_j(z) =  \frac{-1}{2}  R_j^{-n_j} \left( 1 + O \left(\left(\frac{ R_j}{ R_k} \right)^{n_j} \right)\right).$$
	Finally, observe that since $R > 4$ we have $n_k = 2^{N+k-1} \leq 2^{N-1} R^{2^k/2}  \leq R_k^{1/2}$. It follows from Lemma $5.3$ that when $z \in A_k$ we have
	\begin{eqnarray*}
	II &=& \prod_{j=1}^{k-1}\frac{-1}{2}  R_j^{-n_j} \left(1 - 2\left( \frac{ R_j}{ R_k} \right)^{n_j} \right) \\
	&=& \left(\prod_{j=1}^{k-1}\frac{-1}{2}  R_j^{-n_j}(1 + O( R_k^{-1}))\right)\\
	&=& \left(\prod_{j=1}^{k-1}\frac{-1}{2}R_j^{-n_j}\right) (1 + O( R_k^{-1})).
	\end{eqnarray*}
	Finally, by Lemma $7.1$ we have 
	$$III = z^{n_k} F_k(z) = \frac{1}{2} R_k^{n_k} H_{n_k}\left(\frac{z}{R_k}\right).$$
	Combining all these observations yields the desired formula, along with the formulas for $C_k$.
\end{proof}

\begin{lemma}
	Let $R > 8$, $N > 3$, and $k \geq 2$. Then $|C_k| \geq R_{k}^{n_{k-1}}/2^k   \geq 8 R_k.$ When $k = 1$, then $|C_1| = R_1^{n_1}/2  \geq 8 R_1.$
\end{lemma}
\begin{proof}
	The $k =1$ case is obvious. For $k \geq 2$, we can compute, using the fact that $  R_j \leq \sqrt{R_k}$, that
	\begin{eqnarray*}
	|C_k| &=& \frac{1}{2^{k}}  R_k^{n_k} \prod_{j=1}^{k-1} R_j^{-n_j} 
	\geq  \frac{1}{2^{k}}R_k^{n_k} \prod_{j=1}^{k-1}  R_k^{-n_j/2} \\
	&=& \frac{1}{2^{k}}R_k^{n_k - n_{k-1}} 
	= \frac{1}{2^{k}}R_k^{n_{k-1}}. 
	\end{eqnarray*}
	So in this case we see that 
	$$|C_k| \geq \frac{1}{2^k}  R_k^{n_{k-1}} > 8 R_k.$$
\end{proof}

\begin{lemma}
	For all $k$, $C_{k+1} \geq C_k \geq 1.$
\end{lemma}
\begin{proof}
	We compute using the fact that $ R_{k+1} \geq 2  R_k$
	\begin{eqnarray*}
	\frac{|C_{k+1}|}{|C_k|} &=& \frac{1}{2} \frac{ R_{k+1}^{n_{k+1}}}{R_k^{n_{k+1}}} \geq 2^{n_{k+1}-1} \geq 1.
	\end{eqnarray*} 
\end{proof}
The next lemma says that $f$ looks like a power function on $B_k$.
\begin{lemma}
	For $z \in B_k$, we have
	$$f(z) = -C_k \left(\frac{z}{ R_k}\right)^{2n_k} (1 + O( R^{-1}_k)) \cdot (1 + O(4^{-n_{k+1}}))\cdot (1 + O(4^{-n_k})).$$
\end{lemma}
\begin{proof}
	Following the proof of Lemma \ref{A_k}, we have 
	$$f(z) = C_k H_{n_k}\left( \frac{z}{R_k} \right) F_{k+1}(z) (1 + O(R_k^{-1})).$$
	We have kept the $F_{k+1}$ term, since $|z| \geq R_k$. However, $|z| \in B_k$, $4 R_k \leq |z| \leq R_{k+1}/4$, so that,
	\begin{eqnarray*}
	H_{n_k}\left( \frac{z}{ R_k} \right) &=&  \left (\frac{z}{ R_k} \right)^{n_k} \left (2- \left (\frac{z}{ R_k} \right)^{n_k} \right) \\
	&=& \left( \frac{z}{ R_k} \right)^{2n_k} \left(2 \left( \frac{ R_k}{z} \right)^{n_k} - 1\right) \\
	&=& -\left( \frac{z}{ R_k} \right)^{2n_k }  \left(1 + O(4^{-n_k}) \right).
	\end{eqnarray*}
	A similar computation yields
	$$F_{k+1}(z) = \left( 1- \frac{1}{2} \left(\frac{z}{ R_{k+1}}\right)^{n_{k+1}}\right) = (1 + O(4^{-n_{k+1}})).$$
\end{proof}
The following useful corollary is immediate.
\begin{corollary}
	$f$ is never zero on $B_k$.
\end{corollary}

The next proofs follow similar lines of reasoning, and are identical to the proofs in Section 10 of \cite{CB1}. We record them below.
\begin{lemma}
	For all $k\geq 1$, and for $z$ satisfying $5R_k/4   \leq |z| \leq 4 R_k$, we have 
	$$f(z) = C_k \left( \frac{z}{ R_k}\right)^{2n_k}\left(1 + O\left(\left(\frac{4}{5} \right)^{n_k} \right) \right) (1 + O(R_k^{-1})).$$
	\label{Uk}
\end{lemma}

\begin{lemma}
	For $k \geq 1$, and $R_k/4 \leq |z| \leq 4R_k/5$, we have 
	$$f(z) = 2C_k\left( \frac{z}{ R_k}\right)^{n_k} \cdot \left( 1 + O\left( \frac{4}{5} \right)^{n_k} \right)( 1 + O(R_k^{-1})).$$
\end{lemma}

\begin{lemma}
	On $U_k$, we have 
	$$f(z) = C_k \left( \frac{z}{ R_k}\right)^{2n_k} (1 + h_k(z)).$$
	Where $h_k(z)$ is holomorphic on $U_k$ with
	$$|h_k(z)| = O \left( \left(\frac{4}{5}\right)^{n_k} + R_k^{-1} \right).$$
\end{lemma}

\begin{corollary}
	\label{CriticalVk}
	$f'(z)$ is non-zero on $V_k$.
\end{corollary}

We conclude with a remark about the functions $h_k(z)$. First define 
$$\epsilon_k = C \left( \left( \frac{3}{4} \right)^{n_k} + R_k^{-1} \right).$$
The constant $C > 0$ is chosen so that $|h_k(z)| \leq \epsilon_k$ on the annuli $U_k$. It follows that $\sum \epsilon _k$ can be made arbitrarily small, given that $N$ and $R$ are sufficiently large. We need this fact in the proof of Corollary $\ref{CriticalVk}$ and in the proof of Theorem $\ref{C1}$ in the next section. If we assume $N > 10$ is chosen large enough so that these results hold, we may consider such a value $N$ fixed for the remainder of this paper. Moreover, choosing $R$ larger later on will not affect the choice of $N$ needed.

\section{The Mapping Behavior on Annuli}
We record the following useful corollary from the results of the previous section.
\begin{corollary}
	For all $k \geq 1$ we have
	$$\frac{1}{8} \leq \frac{R_{k+1}}{|C_k|\cdot2^{2n_k}} \leq 8.$$
	\label{RkCk}
\end{corollary}
\begin{proof}
	By the first part of the proof of Lemma \ref{A_k} we have for $z \in A_k$ that
	$$f(z) = f_{k}(z) (1 + O(R_k^{-1})).$$
	Recall that $f_{k}$ was the $k$th partial product of the infinite product defining $f$. If we further assume $|z| = 2R_k$, Lemma $7.7$ applies and
	$$\max_{|z|=2R_k} | f_k(z) | = \max_{|z|=2 R_k}  |C_k| \cdot\left( \frac{2R_k}{ R_k}\right)^{2n_k} (1+O(4/5)^{n_k})\cdot(1+O(R_k^{-1})).$$
	By Lemma $2.6$, and perhaps choosing $N$ and $R$ larger, 
	$$\frac{1}{8}   R_{k+1} \leq  |C_k| 2^{2n_k} \leq 8   R_{k+1}.$$
\end{proof}

With Corollary \ref{RkCk} and the estimates of the previous section, the following theorem follows in the exact same way as in Sections $11$ and $12$ of \cite{CB1}.

\begin{theorem}
	\label{itinerary}
	For $N > 10$ and sufficiently large $R$, we have that $A_{k+1} \subset f(V_k) \subset f(A_k)$ and $f(B_k) \subset B_{k+1}$. Moreover, $f$ maps the outermost boundary component of $V_k$ into $B_{k+1}$ and the innermost boundary component into $B_{k}$.
\end{theorem}

\begin{corollary}
	Each set $B_k$ is in the Fatou set of $f$.
\end{corollary}
\begin{proof}
	Since $B_k$ maps into $B_{k+1}$ we know that the iterates tend to infinity uniformly.
\end{proof}

\section{Partitioning the Julia And Fatou Set}
In this section, we describe a system for cataloging the various components of the Julia set and the Fatou set. We will also discuss the conformal mapping properties of $f$, which will be vital to our analysis in the upcoming sections.

Consider the trajectory of a point $z$ with $|z|< R_1/4  $, assuming that $z \notin K_{f}$, the filled Julia set of the polynomial-like mapping $f$. There exists a natural number $M$ so that $|f^{M}(z)| > R_1$, so $z$ is iterated into some annulus $A_1$ or $B_1$. If $z$ is iterated into $B_1$, the point stays in future $B_k$'s for all subsequent iterates. It is also possible that $z$ is iterated into $A_1$ but is still iterated into $B_k$ for some $k$. As discussed earlier, such a $z$ is always in the Fatou set.

The more interesting case is when the orbit of $z$ is disjoint from $B_k$ for all $k$. There are two subcases to consider. For the first case, define 
$$C = \bigcup_{n=0}^{\infty} f^{-n}(K_f).$$
$C$ is the set of all points that eventually map into the filled Julia set of the polynomial-like mapping $f$. If $z$ belongs to $C$, the orbit of $z$ is eventually contained in $K_f$. If $z$ is in the Julia set, the orbit will eventually be contained in $J_f$, and if $z$ is in the Fatou set, the orbit will eventually be contained inside of the attracting basin $B_f$.

The other case is that $z \notin C$. It may not be true that the orbit of $z$ is eventually contained inside of the $A_k$'s for $k \geq 1$. Indeed, the zeros of $f$ are contained in the annular regions $A_k$, so it is possible (and often happens!) that points can be iterated from $A_k$ to $A_j$ for $j < k$. It is possible that points land back in the region between the inner boundary of $A_1$ and $K_f$. It will be useful to label this region by pulling back the annulus $A_1$. First define
$$A_0 = \{|z| \leq 1/4R_1\,:\, f(z) \in A_1\}.$$
$$A_{-k} = \{|z| \leq 1/4 R_1 \,:\, f^N(z) \in |z| \leq 1/4R_1,\, N = 1,2,\dots,k-1,\, f^k(z) \in A_1\}.$$
After $k+1$ steps, $f$ maps $A_{-k}$ into $A_1$. Therefore, we may describe this final case of orbits as the set of all $z$ whose orbit is contained in 
$$A := \bigcup_{k \in \bZ} (A_k \setminus C).$$
From the definition of $A$, it is apparent that $f^{-1}(A) = A$.

The next few lemmas describe the dynamical behavior of points in $A$. 
\begin{lemma}
	Let $W$ be a connected component of $f^{-1}(A_k)$ for $k \in \bZ$. Then we have the following possibilities.
	\begin{enumerate}
		\item $W \subset A_{k-1}$.
		\item $W \subset A_j$ for $j \geq k$. All possible $j$ occur.
	\end{enumerate}
\end{lemma} 
\begin{proof}
By Theorem $\ref{itinerary}$, we know that case $1$ can occur. The fact that all possible $j \geq k$ occurs follows from the fact that all the zeros of $f$ are in the annuli $A_j$ for $j \geq 1$, and the fact that $f$ is continuous. To see that no other cases occur, note that if $z \in A_{j}$ for $j \leq k-2$, then $|f(z)| \leq 1/4 R_k$, since $f(B_{j}) \subset B_{j+1}$ for all $j$.
\end{proof}

The cases above are very different. In the first case listed in the lemma,
$f|_W: W \ra A_k$ is a $2n_k$ to $1$ covering map. This follows from Lemma $\ref{A_k}$. If $W \subset A_j$ for $j > k-1$, then $f|_W: W \ra A_k$ is a conformal mapping. To prove this, we need to show $W$ doesn't contain any critical points. This follows from the following series of observations. First recall that $\Omega_{n_k}^p$ is a ``petal" region where $H_{n_k}(z) = z^{n_k}(2-z^{n_k})$ restricts to be a conformal mapping onto the disk. The following is Lemma $13.1$ in \cite{CB1}:
\begin{lemma}
	\label{JuliaPlace}
	We have $(\cJ(f) \cap A_k )  \subset V_k \cup (R_k \cdot \Omega_{n_k}^p)$ for all $k$. 
\end{lemma} 

The next important lemma says that the diameter of the portion of the Julia set inside the petal $R_k \cdot \Omega_{n_k}^p$ is much smaller than the diameter of that petal. This lemma is deduced on page $35$ of \cite{CB1} as a consequence of the proof Lemma $\ref{JuliaPlace}.$ 

\begin{lemma}
	\label{diam}
	The diameter of the portion of the Julia set contained in each petal $R_k \cdot \Omega_{n_k}^p$ is $O(R_k^{-2} /2^{n_k})$
\end{lemma}

The final lemma characterizes the dynamics of the critical points of $f$. Recall that $B_f$ is the basin of attraction for the polynomial-like mapping $f$. The following is Lemma $14.2$ in \cite{CB1}.

\begin{lemma}
All critical points of $f$ are either contained in $B_f$ or $A_k$ for some $k$. If $z$ is a critical point contained in $A_k$, then $f(z) \in B_k$ and the distance from $z$ to the Julia set is at worst comparable to $R_k/m_k$. In both cases, $z$ is in the Fatou set.
\end{lemma}

We now know enough to show that if $W$ is a component of $f^{-1}(A_k)$ inside of $A_j$ for $j \geq k$, then $f: W \ra A_k$ is a conformal mapping with uniformly bounded conformal distortion. We make this sentence precise. First recall the following consequence of the Koebe distortion theorem.

\begin{lemma}
	\label{Kobe}
	Fix $r < 1$ and let $B = B(0,r) \subset \bD$ be an open ball and let $K \subset B$ be a compact set. Suppose that $f: \bD \ra \bC$ is conformal. Then there exists a constant $C = C_r$, independent of $f$, so that
	$$C_r^{-1} \frac{\diam(f(K))}{\diam(f(B))}\leq \frac{\diam(K)}{\diam(B)} \leq C_r \frac{\diam(f(K))}{\diam(f(B))}.$$	
\end{lemma}

It is easy to deduce this form the the Koebe distortion theorem; for example, one can use \cite{Garnett} Theorem $4.5$ p. $22$. We stress that the constant $C_r$ is independent of $f$ and only depends on the conformal modulus of the annulus $\bD \setminus B$. As $r$ tends to $1$, more distortion is possible.

It is easy to see that we need not apply Lemma $\ref{Kobe}$ to the unit disk $\bD$. Indeed, we may replace $\bD$ by any simply connected domain and $B$ by any simply connected domain compactly contained inside of it. Then the constants of the theorem depend only on the conformal modulus of the annulus between these two domains. When we say bounded conformal distortion, we mean a bound on the conformal modulus, and hence a bound on the constants in Lemma $\ref{Kobe}$.  

\begin{lemma}
	\label{distortion}
The only components $W$ of $f^{-1}(A_k)$ for that are in $A_j$ for $j \geq k$ are in the petals $R_j \cdot \Omega^p_{n_j}$. Moreover, there exists a ball $B$ so that $W$ is compactly contained in $\frac{1}{2}B$ and so that $f|_{2B}$ is conformal.
\end{lemma}
In other words, $f$ maps $W$ to $A_k$ with uniformly bounded conformal distortion. 
\begin{proof}
A proof of the first sentence is found in Lemma $16.2$ of \cite{CB1}. To prove the second fact, recall that any critical point $z \in A_k$ has distance approximately $R_k/n_k$ from the Julia set of $f$. $W$ contains the Julia set inside of the petal $R_j \cdot \Omega^p_{n_j}$ and has diameter at most approximately $R_j^{-2}/n_j$. Therefore, there exists a ball $B$ of unit size so that $2B$ does not contain any critical points, and $1/2B$ contains $W$. 
\end{proof}
Note that the proof above can easily be modified so that if $R$ is initially chosen sufficiently large, we can choose $B$ to make the conformal distortion as close to $1$ as we would like. 

Now we turn to a systematic labeling of the Fatou components of $f$. For $k \geq 1$, define $\Omega_k$ to be the Fatou component containing $B_{k-1}$. Here, we interpret $B_0 = f^{-1}(B_1)$. For $k \leq 0$, define $\Omega_k = (f^{-k+1})^{-1}(\Omega_1)$. The $\Omega_k$'s are distinct; this will follow from the argument at the beginning of the proof of Theorem $\ref{C1}$. 
\begin{lemma}
	\label{OmegaProp1}
	For all $k$, $f(\Omega_k) = \Omega_{k+1}$ 
\end{lemma}
\begin{proof}
If $k \leq 0$ this is true by definition. For $k \geq 1$, we know that $f(B_{k-1}) \subset B_k \subset \Omega_{k+1}$ by Theorem $\ref{itinerary}$. Since $\Omega_{k+1}$ is a connected component of the Fatou set it follows that $f(\Omega_k) \subset \Omega_{k+1}$. But $\Omega_k$ is also a connected component of the Fatou set, so we must have equality.
\end{proof}

Next we want to show that the outermost boundary components of each $\Omega_k$ are $C^1$ smooth, and if $k \geq 1$, they approximate round circles as well. We need the following Lemma, whose proof is in \cite{CB1}.

\begin{lemma}
	\label{AnnulusMap}
	Suppose $h$ is a holomorphic function on $A = \{z \,\colon\, 1 < |z| < 4\}$ and suppose that $|h|$ is bounded by $\epsilon$ on $A$. Let $H(z) = z^m (1+h(z)).$ For any fixed $\theta$ the segment $S(\theta) = \{re^{i\theta}\,\colon\, 3/2 \leq r \leq 5/2\}$ is mapped by $H$ to a curve that makes angle at most $O(\epsilon / m)$ with any radial ray it meets.
\end{lemma}

Recall that by Theorem $\ref{itinerary}$ the image of the annulus $V_k \subset A_k$ contains $A_{k+1}$. It follows from Lemma $\ref{AnnulusMap}$ that $W = f^{-1}(V_{k+1}) \subset V_k$ is a topological annulus, and the boundary components of $W$ are smooth curves that are at most angle $\epsilon_k$ away from being round circles. From here, with the additional help of Lemma $\ref{Uk}$, we can also deduce that the width of $W$ is approximately $R_k / 2n_k$. 

\begin{theorem}
	\label{C1}
	For all $k$, the innermost boundary of $\Omega_k$ is $C^1$ smooth, and if $k \geq 1$, it approximates a round circle. 
\end{theorem}
\begin{proof}
	Fix some $k \geq 1$ and define
	$$\Gamma_{k,n} = \left\{z \in A_k \,\colon\, f^j(z) \in A_{k+j}\, ,\,j = 1,\dots,n\right\}.$$
	Since $A_{k+n}$ is a round annulus, it has a natural foliation of closed circles. $\Gamma_{k,n}$ has an induced foliation of closed analytic curves by pulling back these circles in $A_{k+n}$ via $f$. 
	
	Let $L = L(\theta)$ be a ray through the origin. Then the curves in $\Gamma_{k,n}$ and $\Gamma_{k,n+1}$ intersect this ray with some angle which depends on the particular curve we choose in each family. We let $\phi_n$ and $\phi_{n+1}$ denote the supremum of these angles. By Lemma $\ref{AnnulusMap}$, we have 
	$$|\phi_n - \phi_{n+1}| = O(\epsilon_{n+k}),$$
	where the $\epsilon_n$'s are defined in Section 7.
	
	By our observation before the proof, the widths of these topological annuli $\Gamma_{k,n}$ contract uniformly. Therefore the intersection 
	$$\bigcap_{n=1}^{\infty} \Gamma_{k,n} = \Gamma_k$$
	is some closed set with no interior. We claim that $\Gamma_k$ is the innermost boundary component of $\Omega_{k+1}$. To see that it is in the Julia set, take any open ball with center $z \in \Gamma_k$. For a large enough $n$, we claim that there is a point $w'$ in $A_{k+n}$ so that 
	\begin{enumerate}
		\item $f(w') = 0$
		\item $w'$ has a preimage in $V_{k+n-1}$.
		\item There exists a point $w \in f^{-n}(w') \cap V_k$ that is also in the ball centered at $z \in \Gamma_k$.  
	\end{enumerate}	
	By Theorem $\ref{itinerary}$, we can arrange for $(1)$ and $(2)$ to hold for any $n$. $(3)$ follows from Lemma $\ref{Uk}$. $f^n$ looks like the composition of mappings $z^{n_j}$, so for sufficiently large $n$, there will be a preimage of $w'$ inside of the ball. This point $w \in C$, but $z$ escapes. Since the ball around $z$ was arbitrary, $z$ must be in the Julia set, so $\Gamma_k$ is. To see that it is the innermost boundary component of $\Omega_{k+1}$, again consider any small ball centered at $z \in \Gamma_k$. Then there is a point $w$ in this ball and there exists $n$ so that $w \in \Gamma_{k,n}$ but $w \notin \Gamma_{k,n+1}$ and $w$ is in the unbounded complementary component of $\Gamma$. By Theorem $\ref{itinerary}$, it follows that $f^{n+1}(w) \in B_{k+n+1} \subset \Omega_{k+n+2}$. By Lemma $\ref{OmegaProp1}$, $w \in \Omega_{k+1}$. It follows that $\Gamma_k$ is the innermost boundary of $\Omega_{k+1}$. 

	 Finally we claim $\Gamma_k$ is actually $C^1$ smooth. This follows from the summability of the sequence $\{\epsilon_k\}$. Indeed, the limit $\Gamma_k$ makes at most angle $O(\sum \epsilon_n)$ with the circular arcs that foliate $V_k \subset A_k$, and since the sequence is summable, the tangent directions of the foliation converge uniformly to tangents for $\Gamma_k$. If $k \geq 1$ we can make this curve close to a circle by making the sum $\sum \epsilon_n$ small.  
\end{proof}

It is also true that the outermost boundary component of each $\Omega_k$ is $C^1$ and close to a circle, since it coincides with the innermost boundary of $\Omega_{k+1}$.
\begin{lemma}
	\label{OmegaKProp2}
	The outermost boundary component of $\Omega_k$ is the innermost boundary component of $\Omega_{k+1}$.
\end{lemma}
\begin{proof}
By Theorem $\ref{itinerary}$, the innermost boundary component of $V_k$ maps into $B_k$ for all $n$. It follows that there is a topological annulus $Z_k$ with the same innermost boundary component as $V_k$, and $f(Z_k) \subset B_k \subset \Omega_{k+1}$. By the proof of Theorem $\ref{C1}$, the topological annuli $f^{-n}(Z_{k+n})$ approach $\Gamma_k$. Therefore, nearby any point $z \in \Gamma_k$ is a point $w \in f^{-n}(Z_{k+n})$, for some $n$, so that $f^{n+1}(w) \subset \Omega_{k+n+1}$. It follows from Lemma $\ref{OmegaProp1}$ that $w \in \Omega_k$, so that $z$ is also in the boundary of $\Omega_k$.   
\end{proof}

By Lemma $\ref{diam}$, we know that each $\Omega_k$ is multiply connected (they are actually infinitely connected, which is easy to see from the proof of Theorem $\ref{C1}$ when we showed $\Gamma_k$ was in the Julia set). We can break the complement of $\Omega_k$ into three types of regions $\Omega_k^a$, $\Omega_k^{0}$, and $\Omega_k^{\infty}$. $\Omega_k^0$ is the region containing the origin and $\Omega_k^{\infty}$ is the unbounded region. The remaining regions $\Omega_k^a$ lie between the innermost and outermost boundary components of $\Omega_k$. We define $\Omega_k^A$ to be the union of $\Omega_k$ and all the regions $\Omega_k^a$, so that $\Omega_k^A$ is a topological annulus. We will also consider the regions $\hat \Omega_k$, which we define to be the union of $\Omega_k$, $\Omega_k^0$, and all the regions $\Omega_k^a$. Therefore, $\hat \Omega_k$ is a topological disk. In general, for a set $U$, we denote $\hat U$ as $U$ and the union of all its complementary connected components. In approximation theory, this is often referred to as taking the polynomial hull of $U$. 

Next, let $\Omega_k^a \subset R_k \cdot \Omega_{n_k}^p$ be a complementary component of $\Omega_k$. The boundary of $\Omega_k^a$ is in the Julia set, and $\Omega_k^a$ contains a zero of $f$. By Lemma $\ref{distortion}$, $f$ maps $\Omega_k^a$ conformally onto some disk that contains the origin, and the boundary must also get mapped to the Julia set. Since $0 \in f(\Omega_k^a)$, Lemma $\ref{OmegaKProp2}$ and Lemma $7.2$ show that the boundary gets mapped to the outermost boundary of $\Omega_k$. It follows that inside of $\Omega_k^a$, there are conformal copies of $\Omega_j$ for $j \leq k$. This motivates the following definition.

\begin{definition}
	We call a Fatou component $\omega$ of \textit{k-type} if there exists $m$ so that $f^m: \omega \ra \Omega_k$ is a conformal mapping. That is, $\omega$ maps conformally onto $\Omega_k$.
\end{definition}
The value for $k$ is unique, since $f^{m+1}$ will be an $n$ to $1$ mapping, where $n$ depends on $\Omega_k$. Later, we will prove that every Fatou component that escapes is of $k$-type for some $k$.

There are points in the Julia set that are not in the boundary of any Fatou component. Such points are in the Julia set and are called \textit{buried points}. We say a point $z$ \textit{moves backwards} under $f$ if $z \in A_k$ and $f(z) \in A_j$ for $j \leq k$. $z$ moves backwards infinitely often if there are infinitely many natural numbers $m$ so that $f^m(z)$ moves backwards. 

\begin{lemma}
	\label{movesbackwards}
Let $z \in A$ be given. If $z$ is buried then it moves backwards infinitely often.
\end{lemma}
\begin{proof}
Suppose that $z$ is a buried point. Suppose for the sake of a contradiction $z$ moved backwards only finitely often. By considering an iterate of $z$, we may assume $z$ never moves backwards. $z$ is in the Julia set, so if $z \in A_k$, then $z \in V_k$, and $f^n(z) \in V_{k+n}$ for all $n$. By the proof of Theorem $\ref{C1}$, $z$ must be on the boundary of $\Omega_k$, so it is not buried.
\end{proof}
We will see later that $z$ is buried if and only if it moves backwards infinitely often. For the rest of the paper, we will refer to the points that move backwards infinitely often as $Y$.

\section{The $s$-Sum of the Fatou Components: Preliminaries}
In the next few sections, our goal to prove the following technical theorem, which we need to show that $\Pdim(\cJ(f))$ can be made as close to $s$ as we like. Recall that $s = \Hdim(J_f)$.
\begin{theorem}
	\label{S-Sum}
	Let $\epsilon > 0$ be given. Then $R$ may be chosen large enough so that
	$$\sum_{k=1}^{\infty} \sum_{\omega_k \subset \Omega_1^A} \diam(\omega_k)^{s+\epsilon} < \infty,$$
	where the sum is taken over all Fatou components $\omega_k \subset \Omega_1^A$ of $k$-type for $k \geq 1$.
\end{theorem}
The proof of Theorem \ref{S-Sum} will be carried out in Section $13$ after we prove some preliminary lemmas. For convenience, we will call the sum in Theorem $\ref{S-Sum}$ an \textit{$(s+\epsilon)$}-sum of the components of $k$-type. Recall that the components of $k$-type inside of $\Omega_1^A$ are the escaping Fatou components which eventually iterate conformally onto $\Omega_k$.   

To prove Theorem $\ref{S-Sum}$, we will construct a ``self-improving" covering of the set of points $Y$ which move backwards in the $A_k$'s infinitely often. By this we mean that if at some stage the covering contains a component $\omega_k$ of $k$-type, the new stage will cover all the holes (which we denote $\omega_k^a$, similar to the regions $\Omega_k^a$) of $\omega_k$ and is therefore a subset of the previous stage. We will show that the $(s+\epsilon)$-sum of the new covered components can be compared to the diameter of the previous component in such a way that the $(s+\epsilon)$-sum of the diameters of the components at all stages is summable. The following corollary of such a construction will be immediate.
\begin{corollary}
	\label{Radial}
Let $\epsilon > 0$. Then we may define $f$ in a way that depends on $\epsilon$ so that $\Hdim(Y) \leq s + \epsilon$.	
\end{corollary}

In this section, we will describe this covering before moving on to some technical lemmas we will need in the next sections. The overall idea is that we wait for the first time that a point in $Y \cap A_1$ moves backwards, and refine the covering accordingly. The refinement roughly corresponds to replacing $\omega_k$ by the union of its holes $\omega_k^a$.

Define $W_1^0 = A_1$. For each $z \in A_1 \cap Y$, by definition, there is a first $n$ so that $f^n(z) \in A_k$ for $k < n$. $f^{-n}(A_k)$ has several components in $A_1$, all of which are topological annuli. $z$ is a member of one of these components, which we denote by $W^n_k$.  The collection of all such $W^n_k$ refines the covering of $Y \cap A_1$. We continue to refine the covering by the dynamics as follows. If $z \in W^n_k$, then there is a first $q$ so that $f^{n+q}(z) \in A_j$ for $j < k + q$. We replace $W^n_k$ by the component $W^{n+q}_j \subset (f^{n+q})^{-1}(A_j)$ containing $z$, which is clearly contained inside of $W^n_k$.

It will be useful to make the following modification to the procedure above. Instead of considering all $j < k+q$, we will just consider the $j = k+q-1$ case. To do this, we replace $W^{n+q}_{k+q-1}$ by the polynomial hull of $W^{n+q}_{k+q-1}$, which we denote as  $\hat W^{n+q}_{k+q-1}$. Note that these sets have the same diameter, and they contain all components in the cover of the form $W^{n+q}_{j}$, $j < k+q$ inside of it. Finally, note that $k+q \geq 0$, since we are waiting for the first time the component moves backwards, and $f$ maps $A_k$ with $k \leq 0$ onto $A_{k+1}$ by definition. In other words, the points in the negatively indexed $A_k$ never move backwards.

\section{The $s$-Sum of the Fatou Components: Refining for $k \geq 1$}
In this section, we show that refining the covering in the previous section results in a decreased $(s+\epsilon)$-sum compared to the previous stage, but only for components of $k$-type for $k \geq 1$.  

First we need an easy estimate comparing the diameter of $W^n_k$ to the diameter of the hole inside of it. 
\begin{lemma}
	\label{hole}
$R$ may be chosen so that, for all $k \geq 1$, $\alpha \geq s$, we have
$$\diam(W^n_{k-1})^{\alpha} \leq \frac{1}{4} \diam(W^n_k)^{\alpha}.$$
\end{lemma}
\begin{proof}
There is exactly one component of the form $W^n_{k-1}$ contained inside of the polynomial hull $\hat W^n_k$. By the Koebe distortion theorem
$$\frac{\diam W^n_{k-1}}{\diam W^n_k} \leq C \frac{\diam f^n(W^n_{k-1})}{\diam f^n(W^n_k)} = C \frac{\diam A_{k-1}}{\diam A_k} \leq C \frac{1}{R_0}.$$
Recall that $R_0 := \diam(A_0) = \diam(f^{-1}(A_1)).$ By choosing $R$ large enough, we have the desired result.   
\end{proof}

The next lemma is more complicated. We show that at any stage, when we refine a component $W^n_k$, we can control the diameters of the refined components in terms of the diameter of $W^n_k$.  
\begin{lemma}
	\label{Far}
	$R$ may be chosen so that, for all $k \geq 1$, $\alpha \geq s$
	$$\sum_{q \geq 1} \sum_{W^{n+q}_{k+q-1}} \diam(W^{n+q}_{k+q-1})^{\alpha} \leq \frac{1}{4} \diam(W_k^n)^{\alpha}$$
\end{lemma}
\begin{proof}
	First we need to count how many new components of the type $W^{n+q}_j$ we get for each $q$. First note that by definition of our covering we have the following chain
	\begin{eqnarray*}
		W^{n+q}_{k+q-1} \subset W^n_k &\subset& A_1 \\
		f^n(W^{n+q}_{k+q-1}) \subset f^n(W^n_k) &\subset& A_k \\
		f^{n+1}(W^{n+q}_{k+q-1}) &\subset& A_{k+1} \\
		\vdots &\,& \vdots \\ 
		f^{n+q-1}(W^{n+q}_{k+q-1}) &\subset& A_{k+q-1} \\
		f^{n+q}(W^{n+q}_{k+q-1}) &\subset& A_{k+q-1}. 
	\end{eqnarray*}
	Indeed, we are choosing to cover with the hull of the component that goes back in the $j= n+q-1$ case, and since each time we choose $q$ so it was the first time this happened, we have this exact sequence of mappings. $f$ acts like a covering map in each of these individual situations. There are two possibilities.
	\begin{enumerate}
		\item $f: A_k \ra A_{k+1}$ and $k \geq 0$. In this case, the mapping is $2n_k$ to $1$.
		\item $f: A_k \ra A_{k+1}$ and $k < 0$. In this case, $f$ is a polynomial-like mapping and is $2^N$ to $1$. 
	\end{enumerate}
	It follows from taking preimages according to the definition of the covering that the number of components $W^{n+q}_j$ inside of $W^n_k$ is less than 
	$$ 2^q n_k \cdots n_{k+q-2} \leq 2^q N_{k+q-2}.$$
	Here, $N_{k} = n_1\cdot \cdots n_k$. 
	
	Finally, we can use the last petal map, Lemma $\ref{diam}$, and two applications of the Koebe distortion theorem to conclude that 
	$$\diam(W^{n+q}_{k+q-1}) \leq  \frac{C}{R_{k+q-1}} \diam(W_k^n).$$
	So for each $q$, the contribution to the sum is bounded above by
	$$O \left(2^qN_{k+q-2} R_{k+q-1}^{-1} \right) \cdot \diam (W^n_k)$$
	By Lemma $5.7$, choosing $R$ sufficiently large makes the big-oh term as small as we would like, so the result follows.
\end{proof}

\section{The $s$-Sum of the Fatou Components: Refining for $k \leq 0$}
Having dealt with the refinement of all components of $\Omega_k$-type for $k >0$, we turn to analyzing the refinement with $k \leq 1$. As $k \ra -\infty$, the Fatou components $\Omega_k$ are no longer approximate circles but annuli shaped like the fractal $J_f$. To estimate the $(s+\epsilon)$-sum when we refine the covering in these components, we will decompose them into pieces that map conformally onto a slit $\Omega_1$. The Koebe distortion theorem will allow us to compare the $(s+\epsilon)$ sum of the refinement restricted to one of these pieces to the $(s+\epsilon)$-sum of the refinement of a component of $\Omega_1$-type, with a corrective factor given by the diameter of the piece raised to the $(s+\epsilon)$ power. To control these corrective factors, we will show that they form a Whitney decomposition of the complement of $J_f$, and use Lemma $\ref{VE}$ and the fact that $\Hdim(J_f) = s$ to obtain a convergent sum. The end result is the following:

\begin{lemma}
	\label{Close}
Fix $\epsilon_0 >0$. Let $W^n_1 \in f^{-n}(A_1)$ be an element of the covering of $Y$. Let $W^n_j \in f^{-n}(A_j)$ be the components in the covering contained inside of $\hat W^n_1$. There exists $R$ so that 
	$$\sum_{j=0}^{-\infty} \sum_{W^{n+q}_{j+q-1} \subset W^n_j} \diam(W^{n+q}_{j+q-1})^{s+\epsilon_0} \leq \frac{1}{4} \sum_{W^{n+q}_{j+q-1} \subset W^n_1} \diam(W^{n+q}_{q})^{s+\epsilon_0}.$$
\end{lemma}
\begin{proof}
	For each $\Omega_k$, $k \leq 0$, $f$ is a $2^N$ to $1$ mapping, so each $\Omega_k$ can be decomposed into $2^{N(-k+1)}$ pieces each of which maps conformally onto $\Omega_1$ minus a slit. We denote this slit $\Omega_1$ by $\Omega_1^S$. This process breaks $\Omega_k$ into $2^{N(-k+1)}$ quadrilateral pieces $\cQ = \{Q^k_{i}\}_{i=i}^{2^{N(-k+1)}}$. All of the $Q^k_i$ have holes, and we would like to use these to build a Whitney decomposition of the unbounded complementary component of $J_f$. We denote the filled in components as usual by $\hat \cQ = \{ \hat Q^k_{i}\}_{i=i}^{2^{N(-k+1)}}$; this is just the polynomial hull of each $Q^k_i$. Furthermore, we choose to define $Q^k_i$ by the dynamics of $f$. To be precise, for each $i$ and $k$, we may choose $f(Q^k_i) = Q^{k+1}_{i'}$ for some $i'$. To accomplish this, it suffices to choose an appropriate decomposition of $\Omega_0$, and then define the decomposition of $\Omega_k$ for $k < 0$ by inverse images.
	
	We claim $\hat \cQ$ forms a Whitney decomposition of the complement of the Julia set $J_f$ of the polynomial-like map $f$. Indeed, there exists a quasiconformal mapping $\varphi$ from the complement of $K_f$ to the complement of the disk $\bD$ that conjugates the dynamics of the polynomial like map $f$ with $z^{2^N}$. Under this conjugacy, the cubes in $\hat \cQ$ map to approximate hyperbolic squares that form a Whitney decomposition of the complement of $\bD$ invariant under the dynamics of $z^{2^N}$. It follows that $\hat \cQ$ is a Whitney decomposition by applying $\varphi^{-1}$ since the hyperbolic squares are.
	
	Let us return back to $W^n_1$. Since $\hat \cQ$ forms a Whitney decomposition of the unbounded complement of $J_f$, $f^{-M}(\hat \cQ)$ forms a Whitney decomposition of the unbounded complement of $f^{-M}(J_f)$ contained inside of $\hat W^n_1$. If $W^{n+q}_{j+q-1} \subset W^n_j$, it is contained inside of some quadrilateral $Q = f^{-M}( \hat Q)$, and there exists $W^{n+q}_{q-1} \subset W^n_1$ so that, by the Lemma $\ref{Kobe}$
	$$\frac{\diam(W^{n+q}_{j+q-1})}{\diam(Q)} \leq C \frac{\diam (W^{n+q}_{q-1})}{\diam(W^n_1)}.$$
	Applying this for all $W^{n+q}_{j+q-1} \subset Q$, we obtain
	$$\sum_{W^{n+q}_{j+q-1} \subset Q} \diam(W^{n+q}_{j+q-1})^{s+\epsilon_0} \leq C \frac{\diam(Q)^{s+\epsilon_0}}{\diam(W^n_1)^{s + \epsilon_0}}\sum_{W^{n+q}_{j+q-1} \subset W^n_1} \diam(W^{n+q}_{q})^{s+\epsilon_0}.$$
	Next, if we sum over all the pieces $Q \in f^{-M}(\hat \cQ)$ that make up the decomposition of each Fatou component, we get 
	$$\sum_{j=0}^{-\infty} \sum_{W^{n+q}_{j+q-1} \subset W^n_j} \diam(W^{n+q}_{j+q-1})^{s+\epsilon_0} \leq C \cdot\frac{\sum_{Q \in f^{-M}(\hat \cQ)} \diam(Q)^{s+\epsilon_0}}{\diam(W^n_1)^{s+\epsilon_0}} \sum_{W^{n+q}_{j+q-1} \subset W^n_1} \diam(W^{n+q}_{q})^{s+\epsilon_0}.$$
	Since $f^{-M}(\hat \cQ)$ is a Whitney decomposition of the complement of $f^{-M}(K_f)$, the sum converges by Theorem $\ref{VE}$, and the sum is comparable to $\diam(W_0^n)$. By Lemma $\ref{hole}$, and by choosing $R$ to be sufficiently large we have 
	$$\sum_{j=0}^{-\infty} \sum_{W^{n+q}_{j+q-1} \subset W^n_j} \diam(W^{n+q}_{j+q-1})^{s+\epsilon_0} \leq \frac{1}{4} \sum_{W^{n+q}_{j+q-1} \subset W^n_1} \diam(W^{n+q}_{q})^{s+\epsilon_0}.$$
\end{proof}

\section{The $s$-Sum of the Fatou Components: Conclusions}
By combining the two technical lemmas above, we can now prove the Theorem $\ref{S-Sum}$. The rest of the section is dedicated to some simple corollaries of the proof below and a discussion of the geometry of the Fatou and Julia sets relevant to the next section. 

\begin{proof}[Proof of Theorem $\ref{S-Sum}$ and Corollary $\ref{Radial}$]
It is sufficient to prove that the diameters of all components $W^n_k$ for $k \geq 1$ converge. Indeed, each component $\omega_k$ can be associated to an element $W^n_k$ of the covering of $Y$ since the outer boundary component of $\omega_k$ is contained in a unique $\hat W^n_k$, so that they have comparable diameters.

Suppose we are at the $m$th stage of the refinement procedure described in Section $10$. Let $S$ denote all the collection of sets $\hat W^n_{k}$ obtained at this stage, and $S'$ denote all the components obtained by refining $S$. All of the components in $S'$ contained in $\hat W^n_k$ are contained in $W^n_j \subset \hat W^n_k$ for $j \leq k$. First we write
\begin{eqnarray*}
 \sum_{j=k}^{-\infty} \sum_{W^{n+q}_{j+q-1} \subset W^n_j} \diam(W^{n+q}_{j+q-1})^{s+\epsilon_0} 
	&=& \sum_{j=1}^k \sum_{W^{n+q}_{j+q-1} \subset W^n_j} \diam(W^{n+q}_{j+q-1})^{s+\epsilon_0} \\
	&\,& +  \sum_{j=0}^{-\infty} \sum_{W^{n+q}_{j+q-1} \subset W^n_j} \diam(W^{n+q}_{j+q-1})^{s+\epsilon_0}.
\end{eqnarray*}
We use Lemma $\ref{Far}$ and Lemma $\ref{Close}$ to estimate each of these terms so that 
$$ \sum_{j=k}^{-\infty} \sum_{W^{n+q}_{j+q-1} \subset W^n_j} \diam(W^{n+q}_{j+q-1})^{s+\epsilon_0} \leq  \frac{1}{4} \sum_{j=1}^k \diam(W^n_j) + \frac{1}{4} \diam(W^n_1) $$
Now we can apply Lemma $\ref{hole}$ to see that	
$$\sum_{j=k}^{-\infty} \sum_{W^{n+q}_{j+q-1} \subset W^n_j} \diam(W^{n+q}_{j+q-1})^{s+\epsilon_0} \leq \frac{1}{2} \diam(W^n_k). $$

It follows that the diameters of each refinement decay geometrically, so that the sum of all components of the covering is finite. This proves Theorem $\ref{S-Sum}$. Since the tail of a convergent series tends to $0$, it follows that the sum of the diameters of the $m$th refinement tends to $0$ as $m$ tends to infinity. Since the $m$th refinement covers $Y$, it follows that $H^{s +\epsilon_0}(Y) = 0$, so that $\Hdim(Y) \leq s + \epsilon_0$.
\end{proof}

\begin{corollary}
	$J(f)$ has zero Lebesgue measure.
\end{corollary}
\begin{proof}
	The Julia set is the disjoint union of the the set of points that move backwards infinitely often, countably many $C^1$ curves, and countably many quasicircles with dimension strictly less than $2$. All of these components have zero measure. 
\end{proof}

\section{A Detailed Description of the Dynamics of $f$}
We can now offer a complete description of the Fatou and Julia set, along with several other dynamically interesting facts.

Recall that for an entire function $f$, we define the escaping set as 
$$I(f) = \{z: |f^n(z)| \ra \infty\}.$$
Choose some number $S_0$ so that there exists $z$ with $|z| = S_0$ so that $z \in I(f)$ (for example, choose $S_0$ so that $|z|=S_0 \subset B_1$). Then define inductively
$$S_{n+1} = \max_{|z|=S_n} |f(z)|.$$
We define the fast escaping set as
$$A(f) = \{z: \, \mbox{there exists $k \geq 0$ so that} \, |f^{n+k}(z)| \geq S_n \, \mbox{for all $n \geq 0$}\}.$$
We define the bungee set as 
$$BU(f) = \{z: \mbox{there exists $n_k$ and $n_j$ so that} \, |f^{n_k}(z)| \ra \infty \, \mbox{and} \, f^{n_j}(z) \, \mbox{is bounded}\}.$$
Lastly, we define the bounded orbit set as 
$$BO(f) = \{z: f^n(z) \, \mbox{is bounded}\}.$$

Every point $z \in \bC$ is either in $I(f)$, $BU(f)$, or $BO(f)$. However, it is clear that if $z$ moves backwards infinitely often, then $z$ cannot be in the fast escaping set. 

\begin{corollary}
The sets $I(f)\setminus A(f)$, $BU(f)$, and $BO(f) \setminus C$ all have Hausdorff dimension $\leq s + \epsilon_0<2$.
\end{corollary}

It turns out all the points in the corollary above compose the entirety of of the set $Y$ of points that move backwards infinitely often. 

\begin{corollary}
	\label{movesbackwards2}
A point $z$ moves backwards infinitely often if and only if it is buried.
\end{corollary}
\begin{proof}
By Lemma $\ref{movesbackwards}$, it only remains to show that moving backwards infinitely often implies the point is buried. A point that moves backwards infinitely often cannot be in the set $C$, the set of all preimages of the filled Julia set $K_f$. So suppose that $z$ is on the boundary of some other Fatou component of $f$. Then, this Fatou component must also move backwards infinitely often. By the corollary above, this is impossible. 
\end{proof}

\begin{corollary}
	Every escaping Fatou component is of $k$-type for some unique $k$.
\end{corollary}
\begin{proof}
	Uniqueness has already been discussed. Let $\omega$ be an escaping Fatou component, but suppose it is not of $k$-type. Then its boundary is in the Julia set, and by Lemma $\ref{movesbackwards2}$, the boundary moves backwards finitely often. So it suffices to deal with the case that $\omega$ is a component which never moves backwards. In this case, since the boundary of $\omega$ is in the Julia set and not of $k$-type (and therefore not one of the $\Omega_k$'s), Lemma $\ref{JuliaPlace}$ says that it is in a petal $R_k \cdot \Omega_{n_k}^p$, or in $V_k$. If it is in a petal, it moves backwards, so $\omega$ is in some $V_k$, and must remain in all future $V_{k+j}$'s since it never moves backwards.
\end{proof}

\begin{corollary}
	The boundary of any escaping Fatou component is the union of $C^1$ curves.
\end{corollary}
\begin{proof}
	All escaping components are of $k$-type for some $k$, so it suffices to show this for $\Omega_k$. Since all the boundary components of $\Omega_k$ are escaping by Lemma $\ref{OmegaProp1}$, they are the boundaries of components of $j$-type for some $j$. The components of $j$-type map conformally onto $\Omega_j$, so their boundaries are also $C^1$ smooth, so all of the boundary components of $\Omega_k$ are indeed $C^1$ smooth.
\end{proof}

We can now offer a full description of the Julia set. 

\begin{theorem}
	\label{JuliaSet}
	The Julia set can be decomposed into three pieces
	\begin{enumerate}
		\item The buried points of $f$. Equivalently, the set $Y$ of points which move backwards infinitely often. 
		\item The $C^1$ components that escape to $\infty$. These components are always the boundary component of some Fatou component of $k$-type.
		\item Preimages of the Julia set of the quadratic-like map $f$.
	\end{enumerate} 
\end{theorem}

Note that it follows from this theorem that the Hausdorff dimension of the Julia set is upper bounded by $s + \epsilon_0$. Define the \textit{radial Julia set} to be the set of points $z \in \cJ(f)$ such that there exists a $\delta$ so that for infinitely many $n$, the disk of radius $\delta$ around $z$ may be pulled back univalently via $f^n$ to the component of $f^{-n}(B(z,\delta))$ containing $z$. We define the \textit{hyperbolic dimension} of $f$ to be the supremum of the Hausdorff dimensions of all hyperbolic subsets of the Julia set (see Definition $2.8$ in \cite{RempeRad}). In \cite{RempUrb}, Rempe-Gillen and Urbanski prove that the Hausdorff dimension of the radial Julia set is equal to the Hausdorff dimension of the subset of points in the radial Julia set who have dense orbit in $\cJ(f)$ (p.1982, Theorem $1.4$). In \cite{RempeRad}, Rempe-Gillen proves that the dimension of the radial Julia set coincides with the hyperbolic dimension. 

\begin{corollary}
$\Hdim(BU(f))) \in (s, s+\epsilon_0)$.
\end{corollary} 
\begin{proof}
The set $Y$ of points that go backwards infinitely often is certainly inside of the radial Julia set, and contains slow escaping, bounded orbit, and points in the Bungee set. By Theorem $\ref{JuliaSet}$, the only points that can possible have dense orbit belong to the set $Y$, and it is clear that points in $BO(f)$ and $I(f)$ cannot have dense orbits. Therefore in this setting, the set of points with dense orbits in the radial Julia set coincides with $BU(f)$. This gives the upper bound for $\Hdim(BU(f)))$.

Since the filled Julia set of the polynomial like mapping $f$, $K_f$, is clearly a hyperbolic set, we see that the hyperbolic dimension of $f$ is lower bounded by $s$. By our comments above, it follows that $\Hdim(BU(f)) \geq s$. 
\end{proof}

We conclude the section by recording a lemma describing the nice geometry of the round Fatou components $\Omega_k$, $k \geq 1$. The proof follows from some basic calculations using the fact that $f$ looks like a power mapping on some portions of $\Omega_k$, along with using Lemma $\ref{distortion}$. A similar discussion is found in Section $19$ of \cite{CB1}.

\begin{lemma}[The Shape of Round Fatou components]
	Choose some Fatou component $\Omega_k$, for $k \geq 1$. Define $d_j = 2(n_k + \dots + n_{j-1})$ for $j > k$. Then $\Omega_k$ has the following geometric properties
	\begin{enumerate}
		\item The inner and outer boundary components of $\Omega_k$ are $C^1$ curves arbitrarily close to round circles.
		\item For all $j \geq k$, there are $n_j \cdot 2^{d_j}$ many boundary components of $\Omega_k$ which lie distance approximately $R_k \cdot 2^{-d_j}$ from the outer boundary component of $\Omega_k$. The boundary components are approximately distance $R_kn_j^{-1}2^{-d_j}$ apart from each other and lie on an approximately round circle. Their diameter is $O(R_{j}^{-1})$. All boundary components of $\Omega_k$ arise in this manner for some $j$; we say that such a component is in the $j$th level of $\Omega_k$. 
		\item All of the boundary components of $\Omega_k$ are approximately round circles.
	\end{enumerate}
\end{lemma}

\section{The Packing Dimension of $\cJ(f)$ is $< 2$}
In this section, we prove that the packing dimension of $\cJ(f)$ can be taken to be arbitrarily close to the Hausdorff dimension, and is therefore less than $2$. To accomplish this, we combine the techniques of the previous sections with the techniques used in \cite{CB1} that resulted in an estimate of packing dimension being $1$. Roughly put, we will decompose the compliment of the Julia set into three regions, and estimate the local upper Minkowski dimension using Theorem $\ref{VE}$.  The regions take the three following general forms, the first two being the most straightforward. We have inverse images of the basin $B_f$, and we also multiply connected Fatou components which are ``far away" from the inverse images of $B_f$ in the sense that these components are almost circular. These regions correspond to the components of $\Omega_k$-type (see Definition $9.7$) for sufficiently large $k$. When $k$ is much less than zero, components of $\Omega_k$ we see a third type of region. These components are far from circular, since their boundary curves accumulate onto the fractal boundary of the appropriate inverse image of $\pa B_f$.    

The following result will be useful in applying Theorem $\ref{VE}$. It follows from the results of Sullivan in \cite{DS1} (see Theorems $3$ and $4$). Recall that if $f$ is polynomial-like, $K_f$ denotes its filled Julia set. 
\begin{theorem}
	\label{Sul}
	Let $f: U \ra V$ be a hyperbolic polynomial-like map. Then we have $\Pdim(\pa K_f) = \Hdim(\pa K_f) = \UMdim(\pa K_f)$.
\end{theorem}
In particular, the polynomial-like map $f(z)$ is hyperbolic since its critical point tends to an attracting fixed point, so its packing dimension and upper Minkowski dimension equal the Hausdorff dimension $s$ as well. 

In order to apply Theorem $\ref{VE}$, we need to decompose the Fatou components into simpler pieces. We collect the following lemmas proved in Section $20$ of \cite{CB1}. The first lemma will allow us to break the infinitely connected Fatou components into simpler, annular regions. See Figure $5$.
\begin{lemma}
	\label{Necklace}
	Let $\Omega$ be a bounded open set containing disjoint open subsets $\{\Omega_j\}$ so that $\Omega \setminus \cup_j \Omega_j$ has zero Lebesgue measure. Then for $t \in [1,2]$ we have 
	$$\sum_{Q \in W(\Omega)} \diam(Q)^t \leq \sum_j \sum_{Q \in W(\Omega_j)} \diam(Q)^t$$
	where $W(\Omega)$ represents a Whitney decomposition of $\Omega$. 
\end{lemma}
\begin{figure}[!h]
\centerline{ \includegraphics[height=3in,width=3in]{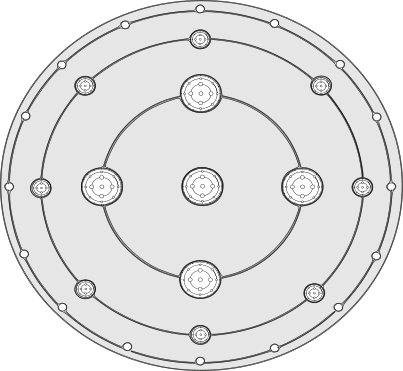} }
\caption{A schematic for the necklacing construction. Holes of $\Omega_k$ in the same circular layer are connected via and approximately circular arcs, and this construction is repeated for all $\omega \subset \Omega_k^A$. The result is that the multiply connected Fatou components are now decomposed into topological annuli, which can be straightened into round annuli by a biLipshitz map. Components of $k$-type for $k \leq 0$ are further broken up into quadrilaterals that map conformally onto $\Omega_k^S$, similar to the proof of Lemma $\ref{Far}$.  Lemmas 14.2 and 14.3 say that it suffices to estimate the critical exponent for a Whitney decomposition of the complement of the ``necklaced" Julia set of $f$.}
\end{figure}

The topological annuli we obtain are close to round annuli if the Fatou component is far enough from the Julia set. This makes calculating the $s$-Whitney sum easy, according to the next two lemmas. 
\begin{lemma}
	\label{BiLip}
	If $f:\Omega_1 \ra \Omega_2$ is biLipschitz, then for any $t \in (0,2]$, we have
	$$\sum_{Q \in W(\Omega_1)} \diam(Q)^t \cong \sum_{Q \in W(\Omega_2)} \diam(Q)^t.$$
\end{lemma}

Given a Whitney decomposition $W$, the $s$-Whitney sum is the sum of the diamters of the elements of $W$ raised to the $s$ power. 
\begin{lemma}
	\label{Round Annulus}
	Let $A(r,r+ \delta)$ be a round annulus and $t \geq 1$. Then the $t$-Whitney sum is 
	$$O\left( \frac{1}{(t-1)} \delta_j^{t-1} r^t \right).$$
\end{lemma}

Recall in Section $9$ that we labeled Fatou components which wind around $0$ as $\Omega_k$, corresponding to which $B_{k-1}$ they contained. As a corollary to the above work, we have the following basic estimate.
\begin{theorem}
	\label{basic est}
	Let $W(\omega_k)$ be a Whitney decomposition for a component $\omega_k$ of $\Omega_k$-type, $k \geq 1$, and let $t \geq s$. Then 
	$$\sum_{Q \in W(\omega_k)} \diam(Q)^{t} = O\left (\frac{1}{(t-1)} \diam(\omega_k)^t \right).$$
\end{theorem}
\begin{proof}
	By Lemma $\ref{distortion}$, it sufficies to just consider the Fatou components $\Omega_k$ which wind around the origin. The layers of holes in the Fatou set $\Omega_k$ lie on Jordan curves that can be chosen arbitrarily close to circles. Connect each of these layers with such a curve, decomposing $\Omega_k$ into approximately round annuli. By Lemma \ref{Necklace}, it suffices to estimate the $s$-sum of the Whitney decomposition of each of these annuli. But since the Jordan curves may be chosen close to circles, the annuli are biLipschitz equivalent to round annuli $A(r,r+\delta)$, where $r$ is the diameter of $\Omega_k$, with biLipshitz constant independent of the topological annulus. Therefore, by Lemmas $\ref{BiLip}$ and $\ref{Round Annulus}$, we have.
	$$\sum_{Q \in W(\Omega_k)} = O\left(\frac{1}{(t-1)}\sum_{\delta_j} \diam(\Omega_k)^t \right) =O\left(\frac{1}{(t-1)} \diam(\Omega_k)^s \right).$$ 
\end{proof}
Note that the above ``necklacing" construction where we decompose the Fatou components into approximately parallel annuli works for any $\Omega_k$ with $k \leq 1$ by pulling back the construction under $f$. Let $\omega \subset \Omega^A_1$ (recall that $\Omega^A_1$ is the topological annulus formed by filling in the holes of $\Omega_1$) be any Fatou component. If $\omega$ is of $k$ type for some $k$, the construction can also be pulled back to $\omega$ via $f$. Hence we have decomposed all Fatou components into topological annuli. Those components that are of $k$-type for $k \geq 1$ are roughly circular annuli with $C^1$ boundary components. However, when $k << 0$, the components of $k$-type will be very close to a copy of the fractal Julia set of the quadratic like map and will approximate the Julia set. The length of these boundary components tends to infinity. 

Let $W(\Omega_1)$ be a Whitney decomposition of the necklaced Fatou component $\Omega_1$ together with the necklaced Fatou components contained inside $\Omega^A_1$. Since the critical exponent of a set is independent of the Whitney decomposition, we assume that $W(\Omega_1)$ is taken to be the dyadic Whitney decomposition.

\begin{theorem}
	The $(s+\epsilon)$-sum of the Whitney decomposition $W(\Omega_1)$ above converges for any $\epsilon > \epsilon_0$, where $s$ the Hausdorff dimension of $\pa K_{f}$. 
\end{theorem}
\begin{proof}
	The Fatou components have three types. Those of $\Omega_k$ type for $k \geq 1$, $\Omega_k$ type for $k < 1$, and those that get iterated to $B_{f}$. Therefore,
	$$\sum_{Q \subset W(\Omega_1)} \diam(Q)^{s+\epsilon} = I + II + III.$$
	Here, $I$ represents the cubes in the components of $\Omega_k$-type for $k \geq 1$, $II$ represents cubes in components of $\Omega_k$-type for $k <1$, and $III$ represents cubes that get iterated into $B_{f}$. We estimate each infinite sum separately.
	
	We have already taken care of $I$ above using Theorem $\ref{basic est}$. Indeed, by Theorem $10.1$, the $(s+\epsilon)$-sum of the components of $\Omega_k$-type converges for any $\epsilon > \epsilon_0$.  
	
	Next we estimate $III$. $\pa K_{f}$ is the Julia set of a hyperbolic polynomial-like mapping. It follows from Theorem $\ref{Sul}$ that the upper Minkowski and Hausdorff dimensions of $\pa K_f$ are both $s$. All other components which map to $\pa K_{f}$ do so conformally, and are contained inside of a component $\omega$ of $\Omega_1$ type. This means that the dimensions of the boundary of the copies of all the $K_f$'s are all $s$. It also implies that the copies of $K_f$ map to $K_f$ with bounded distortion. Since the conformal image of a Whitney decomposition is a Whitney decomposition in the range of the conformal mapping, and since the image cubes have bounded finite overlap with a fixed dydadic Whitney decomposition $W(B_f)$ of the bounded complement of $K_f$, we have for a given copy $B_f'$ of the basin of attraction that 
	$$\sum_{Q \in W(\Omega_1) \cap B_f'} \diam(Q)^{s+\epsilon} \leq C \diam(\omega)^{s+\epsilon} \sum_{Q \in W(B_f)} \diam(Q)^{s + \epsilon}.$$
	Here $\omega$ is the Fatou component of $\Omega_1$-type containing $B'_f$. Summing over all such components, we have
	$$III \leq C \sum_{\omega \subset \Omega^A_1} \diam(\omega)^{s +\epsilon_0} \cdot\sum_{Q \in W(B_f)} \diam(Q)^{s+\epsilon_0}.$$
	Therefore $III$ converges because of Theorem $\ref{S-Sum}$ and Theorem $\ref{VE}$.
	
	Lastly we show $II$ is finite. We apply a technique similar to the one used in the proof of Theorem $\ref{S-Sum}$. Every component of $\Omega_k$-type for $k <1$ is contained in the polynomial hull of a unique Fatou component of $\Omega_1$-type.  We fix such a component and call it $\omega$. We further refine $W(\Omega_1)$ (via Lemma \ref{Necklace}) by taking our necklaced components and slicing them into pieces which map conformally onto $\Omega_1^S$. These pieces are chosen to be the same collection $\hat \cQ = \{\hat Q^k_i\}$ from Theorem $\ref{S-Sum}$, except sliced by the necklacing construction. We pull back this refinement to all other Fatou components being summed over via the conformal mapping $f$. 
	
	Now choose some cube $S$ in the refinement of $W(\Omega_1)$ contained in $\omega \subset \Omega^A_1$, a component of $1$-type. Then we have 
	$$\diam(S)^{s+\epsilon} \leq C \diam(\omega)^{s+\epsilon} \diam(f^M(S))^{s+\epsilon}.$$
	$f^M(S) \subset Q$ for some $Q \in \hat \cQ$, therefore  
	$$\diam(f^M(S))^{s+\epsilon} \leq C \diam(Q)^{s+\epsilon} \cdot \diam(f^{M+k}(S))^{s+\epsilon}.$$
	Combining these estimates we have 
	$$\diam(S)^{s+\epsilon} \leq C \diam(\omega)^{s+\epsilon}\diam(Q)^{s+\epsilon}  \diam(f^{M+k}(S))^{s+\epsilon}.$$

	Next, we sum over all $S \subset f^{-M}(Q)$. This yields
	$$\sum_{S \subset Q} \diam(S)^{s+\epsilon} \leq C \diam(\omega)^{s+\epsilon}\diam( Q)^{s+\epsilon} \sum_{S \subset Q}\diam(f^{M+k}(S))^{s+\epsilon}.$$
	Similarly to the previous case, by the Koebe distortion theorem, the sum on the right hand side is comparable to a fixed $(s+\epsilon)$-sum of some Whitney decomposition of a necklaced $\Omega_1$. The sum therefore converges. Next, we sum over 
	all $Q$ in the collection $f^{-M}(\hat \cQ)$ contained in $\hat \omega$. Therefore,
	$$\sum_{S \subset \hat \omega} \diam(S)^{s+\epsilon} \leq C  \diam(\omega)^{s+\epsilon}\sum_{Q \in \hat \cQ} \diam(Q)^{s+\epsilon}.$$
	The sum on the right hand side converges by Theorem $\ref{VE}$.  Finally, we sum of all possible components $\omega$ of $1$-type to see that
	$$II  \leq C \sum_{\omega \subset \Omega^A_1} \diam(\omega)^{s+\epsilon}.$$
	This sum converges by Theorem $\ref{S-Sum}$. It follows that $II$, converges, proving the theorem.
\end{proof}

\begin{corollary}
	The upper Minkowski dimension, and hence the packing dimension of $\cJ(f) \cap A_1$ is at most $s+\epsilon_0$.
\end{corollary}
\begin{proof}
	The above argument shows that the critical exponent for the Whitney decomposition of $\cJ(f) \cap A_1$ is less than or equal to $s+\epsilon_0$. Since $\cJ(f)$ has zero Lebesgue measure, Theorem \ref{VE} says that the upper Minkowski dimension of $\cJ(f) \cap Omega^A_1$ is also less than or equal to $s+\epsilon_0$. By \cite{RS}, since $f$ has no exceptional values, the local upper Minkowski dimension is constant and coincides with the packing dimension, so that $\Pdim(J(f)) \leq s+\epsilon_0$. 
\end{proof}

\bibliographystyle{alpha}
\bibliography{Pdim}

\end{document}